\documentclass{article}

\usepackage[a4paper,margin=2cm]{geometry} % Simplified geometry setup
\usepackage{amssymb,amsmath,amsthm,color,graphicx,dsfont,mathrsfs,stmaryrd,mathtools}
\usepackage{braket} % For Dirac notation
\usepackage{authblk} % For author affiliations
\usepackage[colorlinks=true,allcolors=blue]{hyperref} % Hyperlinks with colored links
\usepackage[capitalise]{cleveref} % For clever referencing
\usepackage{diagbox}
\usepackage{arydshln}
\usepackage{array,multirow}
\usepackage[
    backend=biber,
    maxbibnames=6,
    maxcitenames=6,
    style=alphabetic,
    sorting=nyt,
    isbn=false,
    doi=false,
    url=false,
giveninits=true, % Changed firstinits to giveninits
date=year]{biblatex}
\usepackage[multiuser,draft,inline,marginclue]{fixme} % For collaborative notes
\FXRegisterAuthor{rr}{arr}{RR}
\FXRegisterAuthor{la}{ala}{LA}
\fxusetheme{colorsig}

\newtheorem{theorem}{Theorem}
\newtheorem{lemma}{Lemma}
\newtheorem{remark}{Remark}
\newtheorem{definition}{Definition}
\newtheorem{proposition}{Proposition}
\newtheorem{corollary}{Corollary}
\newtheorem{example}{Example}
\newtheorem{assumption}{Assumption}

\newcommand{\Ndissip}{N_d}

\newcommand{\N}{\mathbb{N}}

\newcommand{\C}{\mathbb{C}}
\newcommand{\R}{\mathbb{R}}

\renewcommand{\H}{\mathcal{H}}

\newcommand{\D}{\mathcal{D}}

\newcommand{\T}{\mathcal{T}}
\renewcommand{\S}{\mathcal{S}}

\newcommand{\xL}{\mathcal{L}}
\newcommand{\xH}{\mathcal{H}}

\newcommand{\xA}{\mathcal{A}}
\newcommand{\xF}{\mathcal{F}}
\newcommand{\xtr}[1]{\operatorname{Tr}\left( #1 \right)}
\newcommand{\xK}{\mathcal{K}}
\newcommand{\xN}{\mathbb{N}}

\newcommand{\norm}[1]{\left\| #1 \right\|}
\newcommand{\trnorm}[1]{\| #1 \|_1}

\let\bf\boldsymbol

\newcommand{\Id}{\operatorname{\mathbf{Id}}}

\newcommand{\orho}{{\bf \rho}}
\newcommand{\oM}{{\bf M}}
\newcommand{\oa}{{\bf a}}
\newcommand{\ob}{{\bf b}}
\newcommand{\oB}{{\bf B}}

\newcommand{\oL}{{\bf L}}
\newcommand{\oH}{{\bf H}}
\newcommand{\oG}{{\bf G}}

\newcommand{\oP}{{\bf P}}
\newcommand{\oS}{{\bf S}}
\newcommand{\oT}{{\bf T}}

\newcommand{\oX}{{\bf X}}
\newcommand{\oLambda}{{\bf \Lambda}}
\newcommand{\osigma}{{\bf \sigma}}
\newcommand{\oQ}{{\bf Q}}

\newcommand{\oY}{{\bf Y}}
\newcommand{\oA}{{\bf A}}

\newcommand{\oO}{{\bf O}}

\DeclarePairedDelimiter\floor{\lfloor}{\rfloor}

\title{Unconditionally stable time discretization of Lindblad master equations in infinite dimension using quantum channels}

\author[1]{Rémi Robin\thanks{remi.robin@minesparis.psl.eu}}
\author[1]{Pierre Rouchon\thanks{pierre.rouchon@minesparis.psl.eu}}
\author[1,2]{Lev-Arcady Sellem\thanks{lev-arcady.sellem@usherbrooke.ca}}
\affil[1]{Laboratoire de Physique de l'\'Ecole Normale Supérieure, Mines Paris, Inria, CNRS, ENS-PSL, Sorbonne Université, PSL Research University, Paris, France}
\affil[2]{Institut Quantique and Département de Physique, Université de Sherbrooke, Sherbrooke, Québec, Canada}
\begin{document}
\maketitle
\begin{abstract}

We examine the time discretization of Lindblad master equations in infinite-dimensional Hilbert spaces.
Our study is motivated by the fact that, with unbounded Lindbladian, projecting the evolution
onto a finite-dimensional subspace using a Galerkin approximation inherently introduces stiffness, leading
to a Courant--Friedrichs--Lewy type condition for explicit integration schemes. 

We propose and establish the convergence of a family of explicit numerical schemes for time discretization adapted to infinite dimension. These schemes correspond to  quantum channels and thus  preserve the physical properties of quantum evolutions on the set of density operators:
linearity, complete positivity and trace.  Numerical experiments inspired by bosonic quantum codes illustrate the practical interest of this approach when approximating the solution of infinite dimensional problems by that of finite dimensional problems of increasing dimension.
\end{abstract}

\tableofcontents
\section{Introduction}
The Lindblad equation, also known as Gorini-Kossakowski-Sudarshan-Lindblad equation~\cite{lindbladGeneratorsQuantumDynamical1976,goriniCompletelyPositiveDynamical1976,gkls_history},
describes open quantum systems
within the regime of weak coupling to a Markovian environment
\cite{breuerTheoryOpen2006}.
In this setting, the state of an open quantum system is given by
a density operator $\orho$, that is a positive trace-class operator of unit trace, on a separable Hilbert space $\xH$.
The GKSL equation describes the evolution of the density operator $\orho$ as follows:

\begin{align}
    \label{eq_lindblad_full}
\frac{d\orho}{dt} = \xL(\orho)
	\coloneqq -i[\oH, \orho]
	+ \sum_{j} \left( \oL_j \orho \oL_j^\dagger
			- \frac{1}{2} \left(\oL_j^\dagger \oL_j \orho+ \orho \oL_j^\dagger \oL_j\right)  \right),
\end{align}
where $\oH$ is a self-adjoint operator called the Hamiltonian of the system,
and the so-called jump operators
$(\oL_j)_j$
characterize the interaction of the system with its environment. $\xL$ is a super-operator, namely acting on the set of operators on $\xH$, and is called a Lindbladian.

In numerous applications, including quantum optics, quantum chemistry or quantum electrodynamics, the underlying Hilbert space $\xH$ is infinite-dimensional, and both the Hamiltonian $\oH$ and the operators $(\oL_j)_{j}$ are unbounded.
The numerical simulation of \cref{eq_lindblad_full} thus requires discretizing $\xH$
(referred to as space discretization), which means selecting a finite-dimensional subspace, 
as well as discretizing time to solve the evolution equation.

This paper focuses on structure-preserving time discretization schemes, inspired by the principles outlined in \cite{hairerGeometricNumericalIntegration2006}. The semigroup $(e^{t \xL})_{t\geq0}$ is for every $t\geq0$ a quantum channel, see e.g. \cite[Chapter 8]{nielsenQuantumComputationQuantum2010}. A quantum channel is a linear map that is both Completely Positive and Trace-Preserving (CPTP). We are interested in numerical schemes that share this property.

Recently, in the setting of finite-dimensional Lindblad equations,
Lu and Cao introduced time discretization schemes that preserve the set of density operators \cite{caoStructurePreservingNumericalSchemes2024}. However, these schemes are non-linear, unlike the continuous solution of the Lindblad equation. Similarly, \cite{appeloKrausKingHighorder2024} proposes Completely Positive (CP) schemes that can be renormalized to preserve the trace, but this also comes at the cost of non-linearity, again resulting in non-CPTP maps.
To address this issue, we extend a technique introduced in \cite{jordanAnatomyFluorescenceQuantum2016} [Appendix B], enabling the transformation of any CP scheme from the aforementioned references into a CPTP, and thus linear, scheme. Notably, this construction can be applied even to unbounded Lindbladians, resulting in a CPTP scheme that is inherently bounded.

The first contribution of this article is to theoretically investigate these new explicit quantum channel/CPTP schemes. Our main result, presented in \cref{th_onedissipator_firstorder}, establishes that, under reasonable assumptions, the first order variant of these schemes is consistent and stable for any positive time-step, even for the infinite-dimensional Lindblad equation.

In a second time, we study the interest of these schemes for numerical simulations. To this aim, we adhere to a classical approach by employing a Galerkin method.  Specifically, we project the Hamiltonian $\oH$ and the operators $(\oL_j)_{1\leq j\leq \Ndissip}$ to a finite-dimensional Hilbert space $\xH_N$. One advantage of this approach is that it leads to a Lindblad equation
on the finite dimensional Hilbert space $\xH_N$, thus preserving the structure of the evolution.
Limitation and convergence of this method for the Schrödinger equation can be found in the recent preprint \cite{fischer2025quantumparticlewrongbox}; while \textit{a posteriori} error estimate for the Lindblad equation are developed in \cite{etienneyPosterioriErrorEstimates2025}. Note also that several efforts have been undertaken to go beyond the linear space discretization, with techniques including low rank approximation~\cite{lebrisLowrankNumericalApproximations2013,gravinaAdaptiveVariationalLowrank2024} or model reduction~\cite{azouitAdiabaticEliminationOpen2016,leregentAdiabaticEliminationComposite2024}.
Once we have set a finite-dimensional Hilbert space, the task reduces to solving a Lindblad equation on this space, taking the form of a set of ordinary differential equations (ODE).
As one should expect, the fact that the operators $\oH$ and $(\oL_j)_{1\leq j\leq \Ndissip}$ are unbounded
implies that stiffness naturally occurs in this set of ODEs,
as illustrated in \cref{subsubsec:CFL}. Hence, the use of explicit solvers%
\footnote{%
	Note that implicit solvers are rarely used for Lindblad equations because they require solving a high-dimensional linear equation involving the super-operator $\xL_N$ on the space of operators on $\xH_N$.	}
leads to unstable schemes if the dimension of the discretization space is increased without reducing the time-step. We show that this restriction does not apply to our quantum channel schemes.

% Note that naive implementations of implicit schemes are unfortunatly often inaccessible for complexity reason. Namely, in the case $\operatorname*{Dim} \xH_N=N$, the evaluation of the Lindbladian on a density operator $\orho_N \in M_N(\C)$ involves several matrix multiplications of size $N\times N$, hence is $O(N^3)$\footnote{Indeed, as $10 \leq  N\leq 10^5$ for the cases of interest in this article, Strassen's algorithm or other improvements with better scaling are not relevant}. Nevertheless, the Lindbladian is a super-operator, thus is represented by $N^2\times N^2$ matrix. As a consequence, an implicit scheme would require to inverse an $N^2\times N^2$ matrix, which scale roughly in $N^6$. This means that except for small dimension or particularly sparse Lindbladian, the naive inversion approach is not realistic. The authors are not aware of any efficient method to solve this problem but believe that an investigation using, for example, Krylov-type methods could lead to interesting results.

%%%%%%%%%%%%%%%%%%%%%%%%%%%%%%%%%%%%%%%%%%%%%%%%%%%%%%%%%%%%%%%%%%%%%%%%%%%%%%%%%%%%%%%%%%%%%%%%%%%%%%%%%%%

The paper is organized as follows.
In \cref{sec_preliminaries}, we present the prerequisite background and definitions.
\cref{sec-intro_notations} collects notations and definitions used throughout the paper.
\cref{sec:cptp_schemes} presents a general method to obtain explicit schemes of arbitrary order
preserving the structure of Lindblad equations in finite dimension. First, we recall the construction of non-linear schemes introduced in \cite{caoStructurePreservingNumericalSchemes2024}. Then, in \cref{subsec_CPTP_schemes}, we utilize these non-linear schemes and extend the technique outlined in \cite{jordanAnatomyFluorescenceQuantum2016}[Appendix B] to develop (linear) CPTP schemes. The study of these schemes is the central contribution of these paper. \cref{section-definition-lindblad} presents the functional analysis framework
used in the study of Lindblad equations in infinite dimension.
\cref{subsec_H_n} recalls regularity results from the literature, used to establish \textit{a priori} estimates.

In Section \ref{sec_infinite_dim_schemes}, we focus on a Lindblad equation in infinite dimension with an \textit{unbounded} generator. Under appropriate assumptions, we provide, in \cref{lem_apriori_estimate_sans_H}, \textit{a priori} estimates on the continuous solution $\orho_t$. Then, we establish the main result of this paper in \cref{th_onedissipator_firstorder}: the convergence of the CPTP schemes introduced in the section \cref{subsec_CPTP_schemes} toward the continuous solution. The key element is the proof of the consistency of the scheme under the \textit{a priori} estimates, established in \cref{lem_oder1}.

Section \ref{sec_Galerkin_and_num} introduces Galerkin space discretization. We then present an example: the many-photon loss process, where we demonstrate that explicit solvers are subject to a stability condition that relates the size of the discretization space and the time-step. We then show that our quantum channel/CPTP schemes are not affected by this stability condition. Finally, we provide numerical examples that illustrate the robustness of the quantum channel schemes with refined space discretization, and study their numerical efficiency.

\section{Preliminaries}
\label{sec_preliminaries}
\subsection{Notations}
\label{sec-intro_notations}
We fix $\hbar=1$ and work with dimensionless quantities. Besides, we use the following notations:
\begin{itemize}
    \item $\xH$ is a complex separable Hilbert space. Scalar products are denoted using Dirac's bra-ket notation, namely $\ket{x}$ is an element of $\xH$, whereas $\bra{x}$ is the linear form canonically associated to the vector $\ket{x}$.
	    \item Operators on $\xH$ are denoted with bold characters such as $\oa, \ob, \orho$, $\oH$, $\oL$.
    \item $\xK^1$ or $\xK^1(\xH)$ is the Banach space of trace-class operator on $\xH$,
	    equipped with the trace norm defined by:
    \begin{align*}
	    \forall \orho \in \xK^1, \quad \trnorm\orho=\xtr{\sqrt{\orho^\dag \orho}}.
    \end{align*}
The convex cone of positive element of $\xK^1$ is denoted $\xK^1_+$. $\xK_d\subset \xK^1_+$ denotes the convex set of density operators, i.e.,
\begin{align*}
	\xK_d= \left\{ \orho \in \xK^1 \mid \orho^\dag= \orho,\, \xtr{\orho}=1, \, \orho\geq 0 \right\}.
\end{align*}

    \item $\xK^2$ or $\xK^2(\xH)$ denotes the space of Hilbert–Schmidt operators on $\xH$,
	    equipped with the Hilbert-Schmidt norm defined by:
    \begin{align*}
	    \norm{\orho}_2=\sqrt{\xtr{\orho^\dag \orho}}.
    \end{align*}

    \item $B(\xH)$ denote the (Von Neumann) algebra of bounded operators on $\xH$.
	    $\Id$ or $\Id_\xH$ denotes the identity operator and $\norm{ \cdot }_\infty$ is the operator norm induced by the Hilbert norm on $\xH$.
	\item If $\xH,\xH'$ are Hilbert spaces, we denote $B(\xH,\xH')$ the Banach space of linear applications from $\xH$ to $\xH'$ that are continuous for the operator norm. More generally, we use $\mathfrak{L}(\xK,\xK')$ to denote the Banach space of linear applications from one Banach space to another.
    \item If $\oA$ is a unbounded operator on $\xH$, we denote its domain by $\D(\oA)$, and define $\D(\oA^\infty) = \bigcap_{n\geq0} \D(\oA^n)$.
	    We denote $\sigma(\oA)$ the spectrum of $\oA$,
		and $R(\lambda, \oA)=(\lambda \Id - \oA)^{-1}$
	its resolvent for $\lambda \in \mathbb{C} \setminus \sigma(\oA)$.
    % \item When working on the Hilbert space $\xH = L^2(\R,\C)$, we define
	%     the so-called annihilation operator
	%     $\oa = \tfrac1{\sqrt2} (x+\partial_x)$. 
	
	\item For any Banach space $\xK$ and $T>0$, the Banach space of essentially bounded function from $[0,T]$ to $\xK$ is denoted $L^\infty(0, T;\xK)$.
	
	\item When working within the Hilbert space $L^2(\xN,\C)$, we denote its canonical Fock basis by $(\ket{n})_{n\in \xN}$. The annihilation operator $\oa$ and creation operator $\oa^\dag$ act as follows on this basis:
	\begin{align}
		\oa\ket{n+1}=\sqrt{n+1}\ket{n},\quad \oa^\dag \ket{n}=\sqrt{n+1}\ket{n+1}.
	\end{align}
	% \item When working on a tensor product space $\xH= \xH^a \otimes \xH^b$,
		% given two operators $\oX_a$ and $\oX_b$
		% defined respectively in $\xH^a$ and $\xH^b$,
		% we will often alleviate the notations by identifying them respectively to the operators
		% $\oX_a\otimes \Id_{\xH^b}$ and $\Id_{\xH^a} \otimes \oX_b$ on $\xH$.
		% Similarly, we will write $\oX_a \oX_b$ for $\oX_a \otimes \oX_b$.
\end{itemize}
\vspace{1em}

For the evolution of open quantum systems, we use the following conventions:
\begin{itemize}
    \item Let $\oL$ and $\oX$ be linear operators and $\orho\in \xK^1$,
	    we define:
    \begin{align}
	    D[\oL](\orho) &= \oL \orho \oL^\dag
	    			-\frac{1}{2} \oL^\dag \oL \orho
				-\frac{1}{2}\orho \oL^\dag \oL\\
	    D^*[\oL](\oX) &= \oL^\dag \oX \oL
	    			-\frac{1}{2} \oL^\dag \oL \oX
				-\frac{1}{2} \oX \oL^\dag \oL,
    \end{align}
    \item $\mathcal{L}$ denotes the Lindbladian super-operator
	    associated to a Hamiltonian $\oH$
		and a family of jump operators
		$(\oL_j)_{1\leq j\leq N_d}$ on $\xH$,
		acting on elements $\orho$ in (a domain in) $\xK^1$ through
    \begin{align}
        \mathcal{L}(\orho)= -i [\oH,\orho] + \sum_{j=1}^{\Ndissip} D[\oL_j](\orho)
    \end{align}
		where $[\oA,\oB] = \oA\oB - \oB\oA$ denotes the commutator of two operators.
    For unbounded $\mathcal{L}$, we refer to \cref{section-definition-lindblad} for a proper definition
		of the semigroup $(\S_t)_{t\geq0}$
		associated to $\mathcal L$,
		and denote $\orho_t=\S_t(\orho_0)$
		the solution of the dynamical system
		\[ \frac d{d t} \orho_t = \mathcal L(\orho_t)\]
		initialized in a given element $\orho_0\in\xK^1$.

    \item $\mathcal{L}^*$ is formally the adjoint of $\mathcal{L}$;
	    for $\oX$ in (a domain in) $B(\xH)$, it takes the form
    \begin{align}
        \label{eq-formal-lindblad_adjoint}
        \mathcal{L}^*(\oX)= i [\oH,\oX] +\sum_{j=1}^{\Ndissip} D^*[\oL_j](\oX).
    \end{align}
    We refer again to \cref{section-definition-lindblad} for a proper definition
		of the associated semigroup $(\T_t)_{t\geq0}$.
		Note that the semigroups
		$(\T_t)_{t\geq 0}$ on bounded operators
		and
		$(\S_t)_{t\geq 0}$ on trace-class operators
		are related through the following identity,
		where $t\geq 0$,
		$\oX\in B(\xH)$ and $\orho_0\in \xK^1$:
    \begin{align}
        \label{eq-semigroup-dual}
        \xtr{\S_t(\orho_0) \oX}=\xtr{\orho_0 \T_t(\oX)}.
    \end{align}
		Moreover, $(\S_t)_{t\geq 0}$ is called the pre-dual semigroup
		associated to $(\T_t)_{t\geq 0}$,
		since $B(\xH)$ is the dual of $\xK^1(\xH)$
		(where the linear functional associated to $\oX \in B(\xH)$
		is $\orho\mapsto \xtr{\oX\orho}$).

		In the physics literature, the evolution of density operators with $\S_t$
		is called the \emph{Schrödinger picture},
		while the evolution of operators with $\T_t$ is called the \emph{Heisenberg picture}.
	\item Given a Hamiltonian $\oH$ and a family of jump operators $(\oL_j)_j$, we reserve the notation $\oG$ for the non-hermitian operator $\oG=-i\oH-\frac12 \sum_{j=1}^{\Ndissip} \oL_j^\dag \oL_j$.
    %\item For a given self-adjoint operator $\oLambda,D(\oLambda)$ on $\xH$ with $\oLambda \geq \Id$, we define the Hilbert spaces $\xH^{\oLambda}$ equipped with the inner product $\norm{\cdot}_\oLambda$ by
	%\begin{align*}
	%	\braket{\varphi|\psi}_\oLambda=\braket{\oLambda^{1/2}\varphi|\oLambda^{1/2}\psi}.
	%\end{align*}
	%Note that $\xH^{\oLambda}$ is isomorphic to the Hilbert space $\D(\oLambda^{1/2})$ equipped with the graph norm. When the context ensure unambiguity, we denote $\xH^0=\xH \supset \xH^1=\xH^\oLambda \supset \xH^2=\xH^{\oLambda^2} \ldots$ We also introduce $\xH^\infty= \cap_{n\geq 0}\xH^n$ which is dense in all the $\xH^n$ equipped with their topologies.
	%
	%Associated to these spaces, we introduce the scale of Banach spaces
	%\begin{align}
	%	\xK_{\oLambda^k}=\{\orho=\oLambda^{-k/2}\sigma \oLambda^{-k/2} \mid \sigma \in \xK^1(\xH), \, \norm{ \orho}_{\xK_{\oLambda^k}}=\norm{\sigma}_1 \}
	%\end{align}
	%we refer to \cref{app_defK} for the proper definition and main properties of these spaces.
\end{itemize}

\subsection{Time discretization in finite dimension}
\label{sec:cptp_schemes}
\subsubsection{Non-linear positivity-preserving schemes}
\label{subsec:Non-linear positivity preserving schemes}
In this section, we recall the main ideas of the method developed in \cite{caoStructurePreservingNumericalSchemes2024} to obtain
(non-linear) numerical schemes of
arbitrarily high order
preserving the positivity and the trace when the underlying Hilbert space $\xH$ is finite dimensional. In particular, all the linear operators are bounded.

We recall that with our notation
$\oG= -i\oH-\frac{1}{2}\sum_{j=1}^{\Ndissip} \oL^\dag_j \oL_j$, and we can thus write the Lindbladian as
\begin{align}
	\xL(\orho)= \oG \orho + \orho \oG^\dag+ \sum_{j=1}^{\Ndissip} \oL_j \orho \oL_j^\dag.
\end{align}
Introducing the two superoperators
\begin{align}
	\xL_1(\orho)=\oG \orho + \orho \oG^\dag,
	\quad
	\xL_2(\orho)= \sum_{j=1}^{\Ndissip} \oL_j \orho \oL_j^\dag,
\end{align}
we have the splitting $\xL=\xL_1+\xL_2$.
Crucially, both $\xL_1$ and $\xL_2$ generate positivity-preserving semigroups.
The semigroup generated by $\mathcal L_1$ has the simple expression
\begin{align*}
    e^{ t \xL_1}(\orho)=e^{ t \oG} \orho e^{ t \oG^\dag}.
\end{align*}
Note that this formula allows to straightforwardly transform any numerical approximation of the semigroup generated by $\oG$
into a positivity-preserving numerical approximation of the semigroup generated by $\xL_1$.
Indeed, for any approximation of the semigroup $e^{\Delta t\oG}$ by $\oM_{0,\Delta t}$ such that $\norm{e^{\Delta t\oG}-\oM_{0,\Delta t}}_\infty \leq C \Delta t^{N+1}$, we have
\begin{align*}
	\norm{e^{\Delta t\xL_1}(\orho)-\oM_{0,\Delta t}\orho \oM_{0,\Delta t}^\dag}_1
	&= \norm{ e^{\Delta t \oG} \, \orho \, e^{\Delta t \oG^\dag} - \oM_{0,\Delta t}\orho \oM_{0,\Delta t}^\dag }_1 \\
	&\leq \norm{e^{\Delta t \oG}\orho(e^{\Delta t \oG^\dag}-\oM_{0,\Delta t}^\dag)}_1 +\norm{ (e^{\Delta t \oG}-\oM_{0,\Delta t})\orho \oM_{0,\Delta t}^\dag}_1\\
	& \leq \norm{e^{\Delta t \oG}}_\infty \norm{\orho}_1 \norm{e^{\Delta t \oG}-\oM_{0,\Delta t}}_\infty \\
	&+ \norm{e^{\Delta t \oG} -\oM_{0,\Delta t}}_\infty \norm{\orho}_1 \norm{\oM_{0,\Delta t}}_\infty.
\end{align*}
Since $(e^{\Delta t \oG})_{t\geq 0}$ is a contraction semigroup, we have $\norm{e^{\Delta t \oG}}_\infty \leq 1$, and thus $\norm{\oM_{0,\Delta t}}_\infty \leq 1+ C \Delta t^{N+1}$.
As a consequence, for $\Delta t$ small enough, the previous equation leads to
\begin{align}
	\norm{e^{\Delta t\xL_1}(\orho)-\oM_{0,\Delta t}\orho \oM_{0,\Delta t}^\dag}_1 \leq \tilde C \Delta t^{N+1} \norm{\orho}_1,
\end{align}
for some constant $\tilde C$.
Finally, $\orho \mapsto \oM_{0,\Delta t}\orho \oM_{0,\Delta t}^\dag$ is a completely positive map.
To obtain a first order scheme, one can choose the approximation
$\oM_{0,\Delta t}=\Id +\Delta t \oG$.
Indeed, $\norm{e^{\Delta t \oG}-\oM_{0,\Delta t}}_\infty \leq C \Delta t^2$ as
	$$\norm{e^{\Delta t \oG}-(\Id+\Delta t \oG)}_\infty=\norm{ \int_0^{\Delta t} (\Delta t-s)e^{s \oG } \oG^2 \, \mathrm{d}s}_\infty \leq \frac{\Delta t^2}{2} \norm{\oG}_\infty^2,$$
were we used again $\norm{ e^{s\oG}} \leq 1$.
Later in the paper, for technical reasons explained in \cref{rmk_other_treatment_of_H},
we also use the approximation
\begin{equation}
	\oM_{0,\Delta t}= \left(\Id-i\frac{\Delta t}{2}  \oH\right)
		\left(\Id + i\frac{\Delta t}{2}  \oH\right)^{-1}
		\left(\Id- \frac{\Delta t}{2} \sum_{j=1}^{\Ndissip} \oL_j^\dag \oL_j\right)
\end{equation}
which still satisfies
\begin{equation}
	\oM_{0,\Delta t}=\Big(\Id-i\Delta t  \oH\Big)\left(\Id- \frac{\Delta t}{2} \sum_{j=1}^{\Ndissip} \oL_j^\dag \oL_j\right)
	+O\left(\Delta t^2\right)
	=\left(\Id+\Delta t \oG\right) + O\left(\Delta t^2\right).
\end{equation}
This choice amounts to treating separately the Hamiltonian and non-Hamiltonian parts of $\oG$
when approximating the semigroup $e^{t\oG}$.
This splitting strategy takes advantage of the fact that
the propagator $e^{-it\oH}$ of the Hamiltonian part can be approximated by a unitary operator
by the symmetric (1,1)-Padé approximant
$(\Id - i\tfrac t2 \oH)(\Id+i\tfrac t2\oH)^{-1}$.

Going back to the approximation of $e^{t\xL}=e^{t(\xL_1+\xL_2)}$, Duhamel's formula gives
\begin{align}
	\label{eq_duhamel_finite_dim}
    e^{\Delta t \xL}(\orho_0)
	&=e^{\Delta t \xL_1}(\orho_0) +\int_{0}^{\Delta t} e^{(\Delta t -s)\xL_1}(\xL_2(\orho_s)) \, ds\\
	&=e^{\Delta t \xL_1}(\orho_0) + \Delta t \, \mathcal L_2(\orho_0)
		+ \int_{0}^{\Delta t} \left( e^{(\Delta t -s)\xL_1}(\xL_2(\orho_s)) - \mathcal L_2(\orho_0) \right) \, ds.
	\label{eq_duhamel_finite_dim_2}
%	&=e^{\Delta t \xL_1}(\orho_0) + \Delta t \, \mathcal L_2(\orho_0) + O(\Delta t^2)
\end{align}
For $\Delta t$ sufficiently small and $0\leq s \leq \Delta t$, we have
\begin{equation}
\begin{aligned}
	\norm{e^{(\Delta t -s)\xL_1}(\xL_2(\orho_s))- \xL_2(\orho_0)}_1
	&\leq \norm{e^{(\Delta t -s)\xL_1}(\xL_2(\orho_s)-\xL_2(\orho_0))}_1
	+ \norm{ \left( e^{(\Delta t -s)\xL_1} - \operatorname{Id }\right) (\xL_2(\orho_0))}_1\\
	&\leq \norm{e^{(\Delta t -s)\xL_1}}_{\mathfrak{L}(\xK^1)} \norm{\xL_2}_{\mathfrak{L}(\xK^1)} \norm{\orho_s-\orho_0}_1
	+ \norm{e^{(\Delta t -s)\xL_1}-\operatorname{Id }}_{\mathfrak{L}(\xK^1)} \norm{\xL_2}_{\mathfrak{L}(\xK^1)} \norm{\orho_0}_1\\
	&\leq C \norm{\orho_0}_1 \Delta t,
\end{aligned}
\end{equation}
where $\norm{\cdot}_{\mathfrak{L}(\xK^1)}$ stand for the operator norm from $\xK^1(\xH)$ to itself and $C$ is a constant independent of $\Delta t$, $s$, and $\orho_0$.
Thus, the integral on the right-hand side of \cref{eq_duhamel_finite_dim_2} can be neglected within an error of order $\Delta t^2$.
Consequently, for each dissipator $\oL_j$,
we introduce $\oM_{j,\Delta t}=\sqrt{\Delta t} \oL_j$, so that $\Delta t \xL_2(\orho_0)=\sum_{j=1}^{\Ndissip} \oM_{j,\Delta t}\orho_0\oM_{j,\Delta t}^\dag$,
and define a first-order scheme%
%\footnote{We proved the consistency, but the scheme being linear it is also stable}
\begin{align}
	\orho_{n+1}=\xA_{\Delta t}^{1}(\orho_n) \coloneq \oM_{0,\Delta t} \, \orho_n \, \oM_{0,\Delta t}^\dag +\sum_{j=1}^{\Ndissip}\oM_{j,\Delta t} \, \orho_n  \, \oM_{j,\Delta t}^\dag=e^{\Delta t \xL}(\orho_n) + \norm{\orho_n}_1 O(\Delta t^2).
\end{align}
Higher-order schemes can be obtained by iteratively applying Duhamel's formula in \cref{eq_duhamel_finite_dim} and choosing a quadrature rule for the resulting integral on
time
simplexes $0\leq t_1 \leq \ldots \leq t_p \leq \Delta t$ -- see \cite{caoStructurePreservingNumericalSchemes2024} for details. All these schemes share the generic structure
\begin{align}
	\label{eq_abstractA}
	\xA_{\Delta t}(\orho)= \sum_{j=0}^{m} \oM_{j,\Delta t} \, \orho \,\oM_{j,\Delta t}^\dag =e^{\Delta t \xL}(\orho)+ \norm{\orho}_1 O(\Delta t^{p+1}),
\end{align}
where $p$ is the order of the scheme
and the integer $m$ and operators $\oM_{j,\Delta t}$ depend on $p$, the number of jump operators $\Ndissip$ and the choice of quadrature rule.

These schemes are completely positive, yet they do not preserve the trace, which deviates by an error of order $\Delta t^{p+1}$ per time-step.
To address this issue, Cao and Lu introduced in \cite{caoStructurePreservingNumericalSchemes2024} the associated normalized non-linear schemes
\begin{align*}
	\tilde{\xA}_{\Delta t}(\orho)=\frac{\xA_{\Delta t}(\orho)}{\xtr{\xA_{\Delta t}(\orho)}}.
\end{align*}
Note that, despite the fact that these normalized schemes preserve both the positivity and the trace, the lack of linearity means that they are not proper quantum dynamical maps.
In particular, they can lose important properties of quantum dynamical maps, such as contractivity of the nuclear norm of the solution over time.

\subsubsection{Quantum channel schemes}
\label{subsec_CPTP_schemes}
We now present an alternative procedure to transform a linear CP scheme of the form given in \cref{eq_abstractA} into a CPTP scheme; this procedures ensures the preservation of the trace without sacrificing the linearity of the scheme. To achieve this, we employ a method introduced in \cite[Appendix B]{jordanAnatomyFluorescenceQuantum2016}
for the first-order approximation of stochastic Lindblad equations.
We observe that their technique also trivially applies to all orders for deterministic equations.
Let us consider the evolution of the identity operator through the dual evolution of the
non trace-preserving scheme in \cref{eq_abstractA}.
For any $\oO\in B(\xH)$, we have
\begin{equation}
	\xtr{\oO \xA_{\Delta t}(\orho)}=\xtr{\xA^*_{\Delta t}(\oO) \orho},
\end{equation}
where
\begin{equation}
	\xA^*_{\Delta t}(\oO)= \sum_{j=1}^k \oM_{j,\Delta t}^\dag \oO \oM_{j,\Delta t},
\end{equation}
and in particular
\begin{equation}
	\xA^*_{\Delta t}(\Id)= \sum_j \oM_{j,\Delta t}^\dag \oM_{j,\Delta t}.
\end{equation}
It is important to emphasize that for the CP scheme $\xA_{\Delta t}$ to be trace-preserving, it is both necessary and sufficient that $\xA^*_{\Delta t}(\Id)=\Id$.
If $\xA$ is a scheme of order $p$, we also have the estimate
\begin{equation}
	e^{\Delta t \xL^*}(\Id)=\Id=\xA^*_{\Delta t}(\Id) + O(\Delta t^{p+1}).
	\label{eq:accuracy_kraus_sum_id}
\end{equation}
Define $\oS_{\Delta t} =\xA^*_{\Delta t}(\Id)= \sum_j \oM_{j,\Delta t}^\dag \oM_{j,\Delta t}$,
which is a positive operator.
For $\Delta t$ small enough, \cref{eq:accuracy_kraus_sum_id} implies that $\oS_{\Delta t}$
is positive definite, and we can define the positive matrix $\oS_{\Delta t}^{-1/2}$
which also satisfies
\begin{equation}
\oS_{\Delta t}^{-1/2}= \Id + O(\Delta t^{p+1}).
\end{equation}

Substituting $\oM_{j,\Delta t}$ with $\widetilde \oM_{j,\Delta t} = \oM_{j,\Delta t} \oS_{\Delta t}^{-1/2}$
in \cref{eq_abstractA} leads to a Completely Positive Trace-Preserving (CPTP) scheme $\xF_{\Delta t}$
\begin{align}
	\label{eq_abstractA_tilde}
	\xF_{\Delta t}(\orho)=\sum_{j=1}^k \widetilde \oM_{j,\Delta t} \orho \widetilde \oM_{j,\Delta t}^\dag =e^{\Delta t \xL}(\orho)+ \norm{\orho}_1 O(\Delta t^{p+1}),
\end{align}
with now
\begin{align}
	\label{eq_sum_Mk}
	\xF_{\Delta t}^*(\Id)=
	\sum_j \widetilde \oM_{j,\Delta t}^\dag \widetilde \oM_{j,\Delta t}
	&=\oS_{\Delta t}^{-1/2} \, \left( \sum_i \oM_{j,\Delta t}^\dag  \oM_{j,\Delta t} \right)\, \oS_{\Delta t}^{-1/2}= \Id.
\end{align}
\cref{eq_abstractA_tilde} shows that $\xF_{\Delta t}$ has the same order $p$ as the original scheme, while \cref{eq_sum_Mk} shows that it is a CPTP map. Thus, introducing the perturbation of the identity $\oS_{\Delta t}$, allows us to enforce trace preservation in any CP scheme without compromising linearity. 
\begin{remark}
	\label{rmk_bounded_oM_i}
	The fact that $\xF_{\Delta t}$ is a CPTP map implies that the time discretized evolution, like the continuous evolution, contracts the nuclear norm.
	Another important feature is that every $0\leq j \leq \Ndissip$, $\oM_{j,\Delta t}$ satisfy $\oM_{j,\Delta t}^\dag \oM_{j,\Delta t} \leq \Id$ which implies that $\norm{\oM_{j,\Delta t}}_\infty \leq 1$.
	Later, we will consider unbounded operators $\oH$ and/or $\oL_j$ in an infinite dimensional Hilbert space.
	Formally extending the previous notations, all the $\oM_{j,\Delta t}$ might then be unbounded operators, whereas $\widetilde \oM_{j,\Delta t}$ are now ensured to be bounded operators with $\norm{\widetilde \oM_{j,\Delta t}}_\infty \leq 1$.
\end{remark}

\subsection{Quantum Markov semigroups with unbounded generators}
\label{section-definition-lindblad}

Let us now focus on the rigorous definition of the solution of Lindblad equations
in the infinite-dimensional setting.
One can opt for either of two equivalent definitions for the solution to a Lindblad equation: one based on the Von Neumann algebra \(B(\xH)\), as pursued in Chebotarev and Fagnola's works \cite{chebotarevSufficientConditionsConservativity1993,chebotarevSufficientConditionsConservativity1998}, or directly on the Banach space of trace class operators, as initially introduced by Davies \cite{daviesGeneratorsDynamicalSemigroups1979,daviesQuantumDynamicalSemigroups1977}. We follow the former approach.
Let us recall the definition of a Quantum Dynamical Semigroup:

\begin{definition}{}
A quantum dynamical semigroup \((\T_t)_{t \geq 0}\) is a family of operators acting on \(B(\xH)\) that satisfies the following properties:
\begin{itemize}
    \item \(\T_0(\oX)=\oX\) for all \(\oX\in B(\xH)\).
    \item \(\T_{t+s}(\oX)=\T_t(\T_s(\oX))\) for all \(t,s\geq 0\) and \(\oX \in B(\xH)\).
    \item \(\T_{t}(\Id) \leq \Id\) for all \(t\geq 0\).
    \item \(\T_{t}\) is a completely positive map for all \(t\geq0\), meaning that for any finite sequences \((\oX_j)_{1\leq j \leq n}\) and \((\oY_j)_{1\leq j \leq n}\) of elements of \(B(\xH)\), we have \(\sum_{1\leq j,l \leq n} \oY_l^\dag \, \T_t(\oX_l^\dag \oX_j) \,\oY_j \geq 0\).
    \item (normality) for every weakly converging sequence \((\oX_n)_n \rightharpoonup X\) in \(B(\xH)\), the sequence \((\T_t(\oX_n))_n\) converges weakly towards \(\T_t(\oX)\).
    \item (ultraweak continuity) for all \(\orho \in \xK^1\) and \(\oX \in B(\xH)\), we have \(\lim_{t\to 0^+} \xtr{\orho \T_t(\oX)}=\xtr{\orho \oX}\).
\end{itemize}
Furthermore, if \(\T_{t}(\Id) = \Id\) for all \(t\geq 0\), the semigroup is called conservative.
\end{definition}

Now, let us establish the connection between the Lindblad equation and the concept of quantum dynamical semigroup. We will make the following assumptions:
%\begin{assumption}
	%\label{assump_to_defined_Lindblad}
	\begin{itemize}
		\item \(\oG=-i\oH-\frac12\sum_{j} \oL_j^\dag \oL_j\) is the generator of a strongly continuous semigroup of contractions on \(\xH\),
		\item $D(\oG)\subset \cap_j \D(\oL_j)$ and for all \(u \in \D(\oG)\)
		\begin{align}
			\label{eq_L_dagLG}
			\sum_{j=1}^{\Ndissip} \braket{\oL_j u|\oL_j u} +2 \operatorname{Re}\braket{\oG u|u}=0.
		\end{align}
	\end{itemize}
%\end{assumption}

\begin{definition}
	The quantum dynamical semigroup \((\T_t)_{t\geq 0}\) is a solution of the Lindblad equation defined in  \eqref{eq-formal-lindblad_adjoint} if the following weak formulation holds:
	\begin{align}
		\label{eq_weak}
	\Braket{v| \T_t(\oX)u} =
	\Braket {v|\oX u}
		+\int_0^t \left( \Braket{\oG v| \T_s(\oX)u} +\Braket{v| \T_s(\oX)\oG u}+ \sum_{j=1}^{\Ndissip} \Braket{\oL_j v| \T_s(\oX) \, \oL_ju} \right) \mathrm{d}s,
	\end{align}
	for all \(u,v \in \D(\oG)\), \(\oX\in B(\xH)\), and \(t\geq 0\).
\end{definition}
One can show (see \cite[Proposition 2.6]{chebotarevSufficientConditionsConservativity1993}) that a given quantum dynamical semigroup \((\T_t)_{t\geq 0}\) satisfies \cref{eq_weak} if and only if it satisfies the following Duhamel formulation
\begin{align}
	\label{eq_duhamel_adjoint}
	\Braket{v| \T_t(\oX)u} =
\Braket {{e^{t\oG} v}|\oX {e^{tG}}u}
+\sum_{j=1}^{\Ndissip} \int_0^t \Braket{\oL_j e^{(t-s)G}v| \T_s(\oX) \, \oL_j e^{(t-s)\oG}u} ds,
\end{align}
for all \(u,v \in \D(\oG)\), \(\oX\in B(\xH)\), and \(t\geq 0\).\\

Iterating a fixed point procedure on \cref{eq_duhamel_adjoint} leads to a quantum dynamical semigroup $(\T_t^{min})_{t\geq 0}$:
\begin{theorem}[Theorem 3.22 of \cite{fagnolaQuantumMarkovSemigroups1999}]
	\label{th_existence_minimalsemigroup}
	Assume that $\oG$ is the generator of a contraction semigroup, $D(\oG)\subset \cap_j \D(\oL_j)$ and for all \(u \in \D(\oG)\), \cref{eq_L_dagLG} holds. Then, there exists a quantum dynamical semigroup $(\T_t^{\min})_{t\geq 0}$ solution of \cref{eq_weak}. Besides, for
	every quantum dynamical semigroup $(\tilde \T_t)_{t\geq 0}$ solving \cref{eq_weak}, we have
\begin{align*}
	\T_t^{\min}(\Id)\leq \tilde \T_t(\Id)\leq \Id.
\end{align*}
\end{theorem}
If the minimal semigroup is conservative, that is for every $t\geq0$, $\T_t^{\min}(\Id)=\Id$, then $\T_t^{\min}= \tilde \T_t$ for all $t$. In this case, there exists a unique quantum dynamical semigroup solution to the Lindblad equation. Note that, although $\xL^*(\Id)=0$ formally, conservativity of the minimal semigroup cannot be deduced from our assumptions because $\Id$ is not necessary in the domain of $\xL^*$. We also have the following counter-example:
\begin{example}[\cite{daviesQuantumDynamicalSemigroups1977}, Example 3.3]
	The minimal semigroup $(\T_t^{min})_{t\geq 0}$
	of the Lindblad equation $\dot \orho = D[\oa^{\dag2}](\orho)$
	is not conservative.
\end{example}

An important tool to study the conservativity of the minimal semigroup is the following representation of its resolvent, as established in \cite{chebotarevTheoryDynamicalSemigroups1990}.
For each $\lambda >0$, we define the completely positive maps
$P_\lambda \in \mathfrak L\left(B(\xH)\right)$
and
$Q_\lambda\in \mathfrak L\left(B(\xH)\right)$
as follows:
\begin{align}
	\braket{u|P_\lambda(\oX )v}&= \int_0^\infty e^{-\lambda s} \braket{e^{s\oG }u|\oX e^{s\oG }v} ds\\
	\braket{u|Q_\lambda(\oX )v}&= \sum_{j=1}^{\Ndissip} \int_0^\infty e^{-\lambda s} \braket{\oL_j e^{s\oG }u|\oX \oL_j e^{s\oG }v} ds
\end{align}
for $u,v\in \D(\oG )$ and $\oX \in B(\xH)$.
Note that $\norm{P_\lambda}_{\mathfrak L\left(B(\xH)\right)}\leq \frac{1}{\lambda}$;
using \cref{eq_L_dagLG} with an integration by parts also yields $\norm{Q_\lambda}_{\mathfrak L\left(B(\xH)\right)}\leq 1$.
Then, using
\cite[Theorem 3.1]{chebotarevSufficientConditionsConservativity1998},
the resolvent of the minimal quantum dynamical semigroup $(R_\lambda^{min})_{\lambda>0}$, characterized by
\begin{align}
	\braket{u|R_\lambda^{min}(\oX )v}= \int_0^\infty e^{-\lambda s} \braket{u|\T_s^{min}(\oX )v} ds,
\end{align}
satisfies,  for every $\lambda>0$ and $\oX \in B(\xH)$:
\begin{align}
	\label{eq_resolvent_min}
	R_\lambda^{min}(\oX )=\sum_{n=0}^\infty Q_\lambda^n(P_\lambda(\oX )),
\end{align}
with the series converging in the strong topology. Note also that the truncated series
$$R_\lambda^{(m)}(\oX )=\sum_{n=0}^m Q_\lambda^n(P_\lambda(\oX ))$$ obeys the recurrence relation
\begin{align}
	\label{eq_resolvent_rec}
	R_\lambda^{(0)}(\oX )=P_\lambda(\oX ), \quad R_\lambda^{(m+1)}(\oX )=P_\lambda(\oX )+ Q_\lambda(R_\lambda^{(m)}(\oX )).
\end{align}

The conservativity of the minimal semigroup is equivalent to $R_\lambda^{min}(\Id)=\frac{1}{\lambda}\Id$ for every $\lambda >0$. Chebotarev and Fagnola derived several necessary and sufficient conditions for this to hold in \cite{chebotarevSufficientConditionsConservativity1993,chebotarevSufficientConditionsConservativity1998,
fagnolaQuantumMarkovSemigroups1999}. In this context, the Lindbladian $\xL$ is called unital whenever the associated minimal semigroup is conservative.
Finally, the pre-dual semigroup \((\S_t)_{t\geq 0}\) on \(\xK^1\) is defined from \((\T_t)_{t\geq 0}\) through \cref{eq-semigroup-dual}.
Given an initial condition \(\orho_0\in\xK^1\),
we denote \(\orho_t = \S_t(\orho_0)\) the solution of the Lindblad equation.
The conservativity of the semigroup \((\T_t)_{t\geq 0}\) is equivalent to the trace-preserving property of \((\S_t)_{t\geq 0}\) on $\xK^1$.

Davies provided another sufficient and necessary condition \cite[Theorem 3.2]{daviesQuantumDynamicalSemigroups1977} linking conservativity with the domain of the generator of the pre-dual semigroup, i.e., the Lindbladian $\xL$:
\begin{proposition}{\cite[Proposition 3.32]{fagnolaQuantumMarkovSemigroups1999}}
	\label{prop_coreL}
	The linear manifold $\mathcal{U}$ generated by rank one operators
	\begin{align}
		\ket{u}\bra{v}, \qquad u,v\in \D(\oG )
	\end{align}
	is contained in the generator of predual semigroup $(\S_t^{min})$ of $(\T_t^{min})$ and
	\begin{align}
		\xL(\ket{u}\bra{v})=\ket{\oG u}\bra{v}+\ket{u}\bra{\oG v}+\sum_{j=1}^{\Ndissip} \ket{\oL_ju}\bra{\oL_jv}.
	\end{align}
	Besides, the following conditions are equivalent:
	\begin{enumerate}
		\item the linear manifold $\mathcal{U}$ is a core for $\xL$,
		\item the minimal semigroup $(\T_t^{min})$ is conservative.
	\end{enumerate}
\end{proposition}

\subsection{Regularity results}
\label{subsec_H_n}

Inspired by \cite{chebotarevLindbladEquationUnbounded1997} and the recent paper \cite{gondolfEnergyPreservingEvolutions2024}, we introduce some tools to obtain \textit{a priori} estimates. Let us consider a given self-adjoint operator $(\oLambda,D(\oLambda))$ on $\xH$ with $\oLambda \geq \Id$.
We define the Hilbert space $\xH^{\oLambda}=\left(\mathcal D(\oLambda^{1/2}), \norm{\cdot}_{\oLambda}\right)$ where $\norm{\cdot}_{\oLambda}$ is defined from the inner product
	\begin{align*}
		\braket{\varphi|\psi}_\oLambda=\braket{\oLambda^{1/2}\varphi|\oLambda^{1/2}\psi}.
	\end{align*}
When the context ensures unambiguity, we denote $\xH^0=\xH \supset \xH^1=\xH^\oLambda \supset \xH^2=\xH^{\oLambda^2} \ldots$
We also introduce $\xH^\infty= \cap_{k\geq 0}\xH^k$, which is dense in all the $\xH^k$ equipped with their norms.
	
	Associated to these spaces, we introduce a hierarchy of Banach spaces
	of trace-class operators
	\begin{align}
		\xK_{\oLambda^k}=\{\orho=\oLambda^{-k/2}\osigma \oLambda^{-k/2} \mid \osigma \in \xK^1(\xH), \, \norm{ \orho}_{\xK_{\oLambda^k}}=\norm{\osigma}_1 \}
	\end{align}
	Note that the mapping $\osigma \to \oLambda^{-k/2}\osigma \oLambda^{-k/2}$ is an isometry from $\xK^1$ to $\xK_{\oLambda^k}$.
	To prove that an element of $\xK^1$ belongs to $\xK_{\oLambda^k}$, we have the following characterization:
\begin{lemma}
    \label{lemma_carac_Kn}
    Let $\orho \in \xK^1_+$, and $k,r \in \R^+$. The following properties are equivalent:
    \begin{enumerate}
        \item $\orho\in \xK_{\oLambda^k}$ and $\norm{\orho}_{\xK_{\oLambda^k}}=r$, i.e., there exists $\osigma \in \xK^1_+$ with $\norm{\osigma}_1=r$ such that $\orho=\oLambda^{-k/2}\osigma \oLambda^{-k/2}$,
        \item \label{lemma_carac_Kn_item2} $\sup_{n>0} \xtr{n\oLambda^k R(n,-\oLambda^k) \orho}=r<\infty$.
    \end{enumerate}
\end{lemma}
\begin{proof}
	Assume $\orho=\oLambda^{-k/2}\osigma \oLambda^{-k/2}$ with $\osigma \in \xK^1_+$. We have
    \begin{align*}
        \xtr{n \oLambda^k R(n,-\oLambda^k) \orho}&=\xtr{\osigma^{1/2} \oLambda^{-k/2} n \oLambda^k R(n,-\oLambda^k) \oLambda^{-k/2} \osigma^{1/2}}\\
        &=\xtr{\osigma^{1/2} nR(n,-\oLambda^k) \osigma^{1/2}}\leq \xtr{\osigma}
    \end{align*}
    where we used $nR(n,-\oLambda^k)=(\Id+n^{-1} \oLambda^k)^{-1}\leq \Id$. Besides, $nR(n,-\oLambda^k)\xrightarrow[n\to \infty]{} \Id$ in the strong operator topology by well-known properties of the resolvent. Thus, it converges in the weak operator topology. Besides, as weak and $\sigma$-weak topologies coincide on bounded sets,
    \begin{align*}
        \lim_{n\to \infty }\xtr{nR(n,-\oLambda^k) \osigma}=\xtr{\osigma}.
    \end{align*}
    Conversely, assume \cref{lemma_carac_Kn_item2}. Then, for any $\psi \in \xH$ with $\norm{\psi}=1$, we have
    \begin{align*}
        \norm{ \left( n\oLambda^k R(n,-\oLambda^k) \right)^{1/2}\orho^{1/2} \ket{\psi}}^2&\leq \xtr{n\oLambda^k R(n,-\oLambda^k)\orho}\\
		& \leq \sup_{n>0} \xtr{n\oLambda^k R(n,-\oLambda^k) \orho}< \infty.
    \end{align*}
	The sequence $\left(n R(n,-\oLambda^k) \right)^{1/2}$ converges strongly towards $\Id$ while keeping a norm smaller than $1$.
	The sequence $\left( n R(n,-\oLambda^k) \right)^{1/2}\orho^{1/2}\ket{\psi}$ is thus bounded in $\D(\oLambda^{k/2})$. As a consequence, $\orho^{1/2}\ket{\psi}$ belongs to the domain of $\oLambda^{k/2}$. Hence, $\xi= \oLambda^{k/2}\orho^{1/2}$ is a well-defined bounded operator. Using again the $\sigma$-weak convergence of $nR(n,-\oLambda^{k})$, we get
    \begin{align*}
        \xtr{n\oLambda^{k}R(n,-\oLambda^{k}) \orho}= \xtr{\xi^\dag R(n,-\oLambda^{k})\xi}\xrightarrow[n \to \infty]{}\xtr{\xi^\dag \xi}
    \end{align*}
    we deduce that $\xi$ is an Hilbert-Schmidt operator, $\orho=\oLambda^{-1/2} \xi^\dag \xi \oLambda^{-1/2}$ and $\norm{\xi^\dag \xi}_1=r$.

\end{proof}

The forthcoming proposition will serve to establish a connection between the regularity of the state $\orho$ and bounds on $\xL(\orho)$.

\begin{proposition}
    \label{prop_duality_K}
    Let $k\geq 0$. For every $\orho \in \xK_{\oLambda^k}$ and $\oM_1,\oM_2\in B(\xH^k,\xH)$,
	we have $\oM_1 \orho \oM_2^\dag \in \xK^1$ and
        \begin{align}
            \norm{\oM_1 \orho \oM_2^\dag}_1 \leq \norm{\oM_1}_{B(\xH^k,\xH)} \norm{\oM_2}_{B(\xH^k,\xH)} \norm{ \orho}_{\xK_{\oLambda^k}}
        \end{align}
\end{proposition}

\begin{proof}
Since $\orho \in {\xK_{\oLambda^k}}$, there exists $\osigma \in \xK^1$ such that $\orho=\oLambda^{-k/2}\osigma \oLambda^{-k/2}$ and $\norm{\orho}_{\xK_{\oLambda^k}}=\norm{\osigma}_1$.
	Moreover, $\norm{\oM_1\oLambda^{-k/2}}_\infty=\norm{\oM_1}_{B(\xH^k,\xH)}$ and $\norm{\oLambda^{-k/2}\oM_2^\dag}_\infty=\norm{\oM_2 \oLambda^{-k/2}}_\infty=\norm{\oM_2}_{B(\xH^k,\xH)}$,
	which concludes the proof.
\end{proof}
\begin{remark}
	\label{rmk:ambiguous}
	In this paper, we consider $\oM_2$ both as an unbounded operator on $\xH$ and as an element of $B(\xH^k,\xH)$. The notation $\oM_2^\dag$ may lead to ambiguity since it refers to both the unbounded operator from $\D(\oM_2^\dag)$ to $\xH$ and the continuous operator in $B(\xH^*,\xH^{k*})$. However, they coincide on $\D(\oM_2^\dag)$ under the duality induced by the pivot space $\xH$. Henceforth, we adopt the convention that the operator on the right of $\orho\in \xK_{\oLambda^k}$ represents the continuous operator from $\xH$ to $(\xH^k)^*$. Note also that, following the classical literature on Sobolev spaces, we could define the negative Sobolev space $\xH^{-k}\subset \xH$ as $\xH^{k*}$.
\end{remark}

\section{Time discretization in infinite dimension}
\label{sec_infinite_dim_schemes}
\subsection{Infinite dimensional schemes}
In this section, $\oG$ and $\oL_j$ denote the unbounded operators defining our Lindblad equation.
Transposing \cref{subsec_CPTP_schemes}, we define a time discretization scheme $\xF_{\Delta t}$ by
\begin{align}
    \notag \oQ&= \sum_{j=1}^{\Ndissip} \oL_j^\dag \oL_j, &
    \oM_{0,\Delta t}&= (\Id-i\frac{\Delta t}{2}  \oH)(\Id + i\frac{\Delta t}{2}  \oH)^{-1} (\Id- \frac{\Delta t}{2}  \oQ)\\
    \notag \oM_{j,\Delta t}&= \sqrt{\Delta t}\oL_j, &
    \oS_{\Delta t}&= \sum_{j=0}^k \oM_{j,\Delta t}^\dag \oM_{j,\Delta t}= \Id + \frac{\Delta t^2}{4} \oQ^2 ,\\
    \label{eq_def_scheme} \widetilde \oM_{j,\Delta t}&= \oM_{j,\Delta t} \oS_{\Delta t}^{-1/2},&
    \xF_{\Delta t}(\orho)&= \sum_{j=0}^k \widetilde \oM_{j,\Delta t} \orho \widetilde \oM_{j,\Delta t}^\dag.
\end{align}

Our main task is now to prove that, in this infinite-dimensional setting and under suitable assumptions, this scheme is still well-defined
and provides a first-order approximation of the continous solution to the Lindblad equation, which is
the object of our main \cref{th_onedissipator_firstorder}.
Let us start with some basic remarks about the well-posedness of $\xF$. Assuming that the positive operator $\oQ$ defined on $ \cap_{j=1}^{\Ndissip} \D(\oL_j)$ is essentially self-adjoint, $\oS_{\Delta t}^{-1/2}$ is a well-defined bounded operator on $\xH$. In addition, note that
$(\Id-i\frac{\Delta t}{2}  \oH)(\Id + i\frac{\Delta t}{2}  \oH)^{-1}$ is a unitary operator and an approximation of $e^{-i \Delta t \oH}$. In particular,
\begin{align*}
    \widetilde \oM_{0,\Delta t}^\dag \widetilde \oM_{0,\Delta t} =   \oS_{\Delta t}^{-1/2} (\Id- \frac{\Delta t}{2}\oQ)^2  \oS_{\Delta t}^{-1/2} \leq \Id, \qquad&\widetilde \oM_{j,\Delta t}^\dag \widetilde \oM_{j,\Delta t} \leq   \oS_{\Delta t}^{-1/2} \oQ  \oS_{\Delta t}^{-1/2} \leq \Id.
\end{align*}
We deduce that $\widetilde \oM_{0,\Delta t}$ and $\widetilde \oM_{j,\Delta t}$ are well-defined and bounded. In fact, we again have
\begin{align*}
    \sum_{j=0}^k \widetilde \oM_{j,\Delta t}^\dag \widetilde \oM_{j,\Delta t}=\Id.
\end{align*}
Thus $\xF_{\Delta t}$ is a CPTP map on $B(\xH)$.

\begin{remark}
    \label{rmk_other_treatment_of_H}

The results of this section (including \cref{th_onedissipator_firstorder}) can be straightforwardly adapted
	to accommodate the
    following variations of the scheme $\xF_{\Delta t}$:
\begin{itemize}
    \item Replace $\oM_{j,\Delta t}$ for $1 \leq j \leq \Ndissip$ by $(\Id-i\frac{\Delta t}{2}  \oH)(\Id + i\frac{\Delta t}{2}  \oH)^{-1}\oM_{j,\Delta t}$, that is performing the splitting consisting of first computing the dynamics corresponding to the non-Hamiltonian part, then applying the Hamiltonian evolution.
    \item Remove $(\Id-i\frac{\Delta t}{2}  \oH)(\Id + i\frac{\Delta t}{2}  \oH)^{-1}$ from $\oM_{0,\Delta t}$ but applies it to all $\oM_{j,\Delta t}$ for $0 \leq j\leq N$ on the right. This corresponds to the first order splitting starting with the Hamiltonian part.
\end{itemize}
Crucially, in all these variations, $\oS_{\Delta t}$ does not depend on $\oH$.
In fact, without extra assumptions, we are not able to replace $\oM_{0,\Delta t}$ by the
	(perhaps more natural) choice $(\Id+\Delta t \oG)$ as this would imply that $\oS_{\Delta t}$ depends on $\oG^\dag \oG$ instead of $\oQ^2$, and $\oQ^2$ is not controlled by $\oG^\dag \oG$ which causes issues to adapt step 3 of \cref{sec_proof_estimat_eq_M1}. 
\end{remark}

\subsection{Assumptions and some \textit{a priori} estimates}
\subsubsection{Simplified assumption when there is no Hamiltonian}
We start with the case $\oH=0$.

\begin{assumption}
    \label{assump_sansH}
    We consider a finite set of closed operators $(\oL_j)_{1 \leq j \leq \Ndissip}$ such that
    \begin{enumerate}
        \item $\cap_j\D(\oL_j^\dag \oL_j)$ is dense in $\xH$ and the operator $\sum_{j=1}^{\Ndissip} \oL_j^\dag \oL_j$ defined on $\cap_j\D(\oL_j^\dag \oL_j)$ is essentially self-adjoint. We denote $(\oQ,\D(\oQ))$ its closure. Additionally, we define $\oG=-2\oQ$ with domain $ \D(\oG)=\D(\oQ)$.
        \item \label{item_apriorisansH}
        We consider the positive operator $\oLambda = \oQ+\Id >0$, and introduce, as in \cref{subsec_H_n}, the Hilbert space $\xH^n$, that is $\D(\oLambda^{n/2})\cong \D(\oG^{n/2})$ with inner product $\braket{\oLambda^{n/2}\cdot |\oLambda^{n/2}\cdot }$. We also assume that $\oL_j$ belong to $B(\xH^5,\xH^4)$ and there exists a constant $\gamma\geq 0$ such that for all $u\in \xH^6$
        \begin{align}
            \label{apriori_sans_H}
            2\operatorname*{Re}\braket{\oLambda^2 \oG u| \oLambda^2 u}+ \sum_{j=1}^{\Ndissip} \|\oLambda^2\oL_ju\|^2\leq \gamma \|\oLambda^2u\|^2
        \end{align}
    \end{enumerate}
    
\end{assumption}

Formally, \cref{apriori_sans_H} means
$\xL^*(\oLambda^4)\leq \gamma \oLambda^4$. This assumption is the key element to prove \textit{a priori} estimates.

Let us check that \cref{assump_sansH} implies that the associated Lindblad equation admits a minimal semigroup. Indeed, for every $u\in \cap_j\D(\oL_j^\dag \oL_j)$, we have from the definition of $\oG$
    \begin{align*}
        2 \operatorname*{Re}\braket{\oG u|u}=-\sum_{j=1}^{\Ndissip}\|\oL_j u\|^2.
    \end{align*}
    As the $\oL_j$ are closed, this implies that $\D(\oG)\subset \cap_{j=1}^{\Ndissip} \D(\oL_j)$. Besides, $\oG=-2\oQ$ is self-adjoint and dissipative, thus the generator of a semigroup of contractions. Hence, there exists a minimal quantum dynamical semigroup solution of the Lindblad equation by \cref{th_existence_minimalsemigroup}.

    Note also that, by definition of the spaces $\xH^n$, $\oQ \in B(\xH^{n+2},\xH^n)$ for all even number $n$. On the other hand, $\oL_j^\dag \oL_j \leq \oQ$ implies that $\oL_j \in B(\xH^1,\xH)$. Together with the assumption that $\oL_j\in B(\xH^5,\xH^4)$, we get by interpolation (see \cref{app_interpol}) that $\oL_j \in B(\xH^{n+1},\xH^{n})$ for $0\leq n \leq 5$.
\subsubsection{Assumption for the general case}
When $\oH\neq0$, there is no longer a canonical choice of reference operator $\oLambda$. We assume that one can find a self-adjoint operator $\oLambda\geq \Id$ such that
    \begin{assumption}
        \label{assump_apriori}
        The operator $\oLambda\geq \Id$ is self-adjoint and
    \begin{enumerate}
        \item \label{asp1} The operators $(\oL_j,\xH^\infty)$ are closable operators for $1 \leq j \leq \Ndissip$,  $\sum_j \oL^\dag_j \oL_j$ and $\oH$ are essentially self-adjoint on $\xH^\infty$ and relatively bounded with respect to $\oLambda$,
        \item \label{asp2} $\oH \in B(\xH^6,\xH^4)$, $\oL_j \in B(\xH^5,\xH^4)$, and $\oQ \in B(\xH^6,\xH^4)$,
        \item  \label{asp3} There exists $\gamma \geq 0$ such that $\forall u \in \xH^6$, \cref{apriori_sans_H} holds:
        \begin{align*}
                %\label{apriori_avec_H}
                2\operatorname*{Re}\braket{\oLambda^2 \oG u| \oLambda^2 u}+ \sum_{j=1}^{\Ndissip} \|\oLambda^2\oL_ju\|^2\leq \gamma \|\oLambda^2u\|^2.
        \end{align*}
        \item \label{asp4} The closure of the operator $(\oG,\xH^\infty)$ is the generator of a semigroup of contractions on $\xH$. Besides, $\xH^4$ is $\oG$-compatible, in the sense that the restriction of the semigroup $(e^{t\oG})_{t\geq 0}$ to $\xH^4$ is a strongly continuous semigroup in $\xH^4$. Finally, $\xH^6$ is stable under the application of the semigroup $(e^{t\oG})_{t\geq 0}$.
    \end{enumerate}
\end{assumption}
\begin{remark}
    \begin{itemize}
        \item \cref{asp1} implies that there exists $\nu>0$ such that $\oLambda \geq \nu \sum_k \oL_k^\dag \oL_k$. As a consequence $\oL_k \oLambda^{-1/2} \in B(\xH)$, i.e. $\oL_k \in B(\xH^1,\xH)$. Similarly $\oH,\oQ \in B(\xH^2, \xH)$.
        \item Using \cref{asp1,asp2} and interpolation (see for e.g. \cref{lem_interpol} of \cref{app_interpol}), one has $\oL_k \in B(\xH^{n+1},\xH^{n})$ for $1\leq n \leq 4$ and $\oH,\oQ \in B(\xH^{n+2},\xH^{n})$ for $0\leq n \leq 4$
        \item \cref{asp4} is a technical assumption. Dissipativity of $\oG$ in $\xH$ and of $\oG-\gamma \Id$ in $\xH^4$ follows from \cref{apriori_sans_H}, but we require maximal dissipativity to ensure existence of a quasi-contraction semigroup.
    \end{itemize}
    
\end{remark}
Under \cref{assump_apriori}, the existence of the minimal semigroup follows from \cref{th_existence_minimalsemigroup}.
\subsubsection{\textit{A priori} estimate}
We have the following \textit{a priori} estimate:
\begin{lemma}
    \label{prop_regularity}
    Under \cref{assump_sansH} or \cref{assump_apriori}, the minimal semigroup $(\T_t^{min})_{t\geq 0}$ is conservative and the following \textit{a priori} estimate holds for every $ t \geq 0$:
        \begin{align}
            \label{eq_apriori_withoutH}
            \sup_n \braket{u |\T_t^{min}(n\oLambda^4R(n,-\oLambda^4))u} \leq e^{\gamma t}\| \oLambda^2 u \|^2, \quad \forall u \in \xH^4,
        \end{align}
        where $\gamma \geq 0$ is the constant defined in \cref{assump_sansH} or \cref{assump_apriori}.
\end{lemma}
The proof primarily involves establishing estimates by utilizing the representation of the resolvent of the minimal semigroup, as provided in \cref{eq_resolvent_min}. We defer this proof to \cref{sec_proof_prop_regularity_sans_H}.

%\begin{remark}
%    In \cref{item_apriorisansH} of \cref{assump_sansH}, $\xH^5$ can be replaced by any core $D$ of $\oLambda^2$, invariant by $e^{t\oG}$ and such that $\oL_j D\subset \xH^4$.
%\end{remark}

\begin{proposition}
    \label{lem_apriori_estimate_sans_H}
    Under \cref{assump_sansH} or \cref{assump_apriori}, the restriction of the Lindblad semigroup $\S_t$ to $\xK_{\oLambda^4}$ is a $\gamma$-quasi contraction semigroup on $\xK_{\oLambda^4}$
\end{proposition}
\begin{proof}
    As $\oS_t$ is positivity preserving, without loss of generality we can restrict our attention to $\orho_0\in \xK^1_+ \cap \xK_{\oLambda^4}$. Let $\osigma_0\in \xK_+^1$ such that $\orho_0=\oLambda^{-2}\osigma_0 \oLambda^{-2}$. Diagonalization of $\osigma_0$ gives a sequence of orthonormal vector $(\mu_k)_k \in \xH^\N$ and a sequence of positive reals $(\mu_k)\in L^1(\N)$ such that $\osigma_0= \sum_{k=0}^\infty \mu_k \ket{u_k}\bra{u_k}$, where the series converges in trace norm. Defining the truncated sequences $\osigma_{0,m}=\sum_{k=0}^m \mu_k \ket{u_k}\bra{u_k}$ and using \cref{prop_regularity}, we have
    \begin{align}
        \xtr{\T_t^{min}(n\oLambda^4R(n,-\oLambda^4)) \oLambda^{-2}\osigma_{0,m}\oLambda^{-2}}&=\sum_{k=0}^{m} \braket{u_k|n\oLambda^4R(n,-\oLambda^4)u_k}\\
        &\leq e^{\gamma t} \sum_{k=0}^m\|u_k\|^2\\
         &\leq e^{\gamma t} \|\osigma_0\|_1
    \end{align}
    Besides, as $\osigma_{0,m}\xrightarrow[m \to \infty ]{\xK^1}\osigma_0$, the increasing sequence $\oLambda^{-2}\osigma_{0,m} \oLambda^{-2}$ converges towards $\orho_0$ in $\xK_{\oLambda^4}$ thus in $\xK^1$, leading to
    \begin{align}
        \xtr{n\oLambda^4R(n,-\oLambda^4)\orho_t} = \xtr{\T_t^{min}(n\oLambda^4R(n,-\oLambda^4)) \orho_0}\leq e^{\gamma t} \|\orho_0\|_{\xK_{\oLambda^4}}.
    \end{align}
    Finally, \cref{lemma_carac_Kn} implies that $\forall t \geq 0$, $\orho_t \in \xK_{\oLambda^4}$ and $\|\orho_t\|_{\xK_{\oLambda^4}} \leq e^{\gamma t} \|\orho_0\|_{\xK_{\oLambda^4}}$.
\end{proof}
\subsection{Main Theorem}
Let us now state our first main theorem
\begin{theorem}
    \label{th_onedissipator_firstorder}
    Assume that \cref{assump_sansH} or \cref{assump_apriori} is satisfied.
    Then, there exists $C>0$ such that for every final time $T \geq 0$, every $(\Delta t,N)\in \R^+ \times \xN$ such that $\Delta t N =T$, we have for every $\orho_0 \in \xK^1_+ \cap \xK_{\oLambda^4}$,
    \begin{align}
        \|\orho_T-\xF_{\Delta t}^N(\orho_0)\|_1 \leq C T \|\orho\|_{L^\infty(0, T;\xK_{\oLambda^4})} \Delta t.
    \end{align}
\end{theorem}
Namely, the time-discretized scheme given by $\xF_{\Delta t}$ is a first order approximation scheme even in the infinite dimensional setting.

The proof relies on the following Lemma whose proof is the object of \cref{subsec_proof_lemma1}.
\begin{lemma}[Consistency]
    \label{lem_oder1}
    Under \cref{assump_sansH} or \cref{assump_apriori}. There exists a constant $C>0$, such that for every final time $T>0$, initial state $\orho_0\in \xK^1_+\cap \xK_{\oLambda^4}$, $\Delta t$ and $0 \leq t \leq T-\Delta t$, we have
    \begin{align}
        \|\orho_{t+\Delta t}- \xF_{\Delta t}(\orho_{t})\| \leq C T \|\orho\|_{L^\infty(0, T;\xK_{\oLambda^4})} \Delta t^2.
    \end{align}
\end{lemma}

To prove the convergence of the scheme stated in \cref{th_onedissipator_firstorder} we need both consistency of the scheme (\cref{lem_oder1}) and a stability property. For the latter, we use that $\xF_{\Delta t}$, being a CPTP map, induces contraction of the trace norm:
\begin{align}
    \|\orho_T-\xF_{\Delta t}^N(\orho_0)\|_1&=\|\sum_{k=1}^{N} \xF_{\Delta t}^{N-k}(\orho_{k\Delta t}-\xF_{\Delta t}(\orho_{(k-1)\Delta t} ))\|_1\\
    & \leq \sum_{k=1}^{N} \|(\orho_{k\Delta t}-\xF_{\Delta t}(\orho_{(k-1)\Delta t} ))\|_1\\
    & \leq C T \|\orho\|_{L^\infty(0, T;\xK_{\oLambda^4})} \Delta t,
\end{align}
which concludes the proof of \cref{th_onedissipator_firstorder}.

Under \cref{assump_sansH} or \cref{assump_apriori}, the bound we proved on the accuracy of the scheme increases exponentially with the final time, namely in $Te^{\gamma T}$ as the upper bound we have on $\|\orho\|_{L^\infty(0, T;\xK_{\oLambda^4})}$ increases exponentially. Under a stronger version than \cref{apriori_sans_H}, we can obtain only a linear dependency in $T$.
\begin{corollary}
    Assume \cref{assump_sansH} or \cref{assump_apriori} holds, and that there exists $\gamma_0,\gamma_1 >0$ such that $\forall u \in \xH^6$,
        \begin{align}
            \label{eq_control_for_uniform_K4}
                2\operatorname*{Re}\braket{\oLambda^2 \oG u| \oLambda^2 u}+ \sum_{j=1}^{\Ndissip} \|\oLambda^2\oL_ju\|^2\leq \gamma_0\|u\|^2 -\gamma_1 \|\oLambda^2u\|^2
        \end{align}
        then \cref{th_onedissipator_firstorder} holds and one has

        \begin{align}
            \|\orho_T-\xF_{\Delta t}^N(\orho_0)\|_1 \leq C T \Delta t \max \left(\|\orho_0\|_{\xK_{\oLambda^4}}, \frac{\gamma_0}{\gamma_1} \right).
        \end{align}
\end{corollary}
\begin{proof}[Sketch of proof]
    We only have to show that \cref{eq_control_for_uniform_K4} implies that for every $\orho_0 \in \xK_{\oLambda^4}$ and $t \geq 0$, $\|\orho_t\|_{\xK_{\oLambda^4}}\leq \max \left(\|\orho_0\|_{\xK_{\oLambda^4}}, \frac{\gamma_0}{\gamma_1} \right)$. Formally \cref{eq_control_for_uniform_K4} means that $\xL^*(\oLambda^4)\leq \gamma_0 - \gamma_1 \oLambda^4$, thus $t\mapsto \xtr{\orho_t \oLambda^4}$ is decreasing if $\xtr{\orho_t \oLambda^4}> \frac{\gamma_0}{\gamma_1}$. We refer to \cite[Section 4.4.3]{robinConvergenceBipartiteOpen2024} for the proof of this statement.
\end{proof}

We also have a non-quantitative convergence for a non-regular initial state:
\begin{corollary}
    Under \cref{assump_sansH} or \cref{assump_apriori}, for any final time $T>0$, for every $\orho_0\in \xK^1_+$,
    \begin{align}
        \xF_{T/N}^N(\orho_0) \xrightarrow[N\to \infty]{\xK^1}\orho_T
    \end{align}
\end{corollary}
\begin{proof}
    For every $\epsilon>0$, by density of $\xK_{\oLambda^4}$ in $\xK^1$, there exists $\tilde \orho_0\in \xK^1_+\cap \xK_{\oLambda^4}$ such that $\|\tilde \orho_0-\orho_0\|_1\leq \epsilon/3$. By \cref{th_onedissipator_firstorder}, there exists $N$ such that $\|\S_t(\tilde \orho_0)-\xF^N_{T/N}(\tilde \orho_0)\|_1\leq \epsilon/3$. Moreover, as both $\xF$ and $\S_t$ are CPTP, they contract the trace norm so
    $\|\S_t(\tilde \orho_0)-\S_t(\orho_0)\|_1 \leq \epsilon/3$ and 
    $\|\xF^N_{T/N}(\tilde \orho_0)-\xF^N_{T/N}(\orho_0)\|_1 \leq \epsilon/3$.
\end{proof}
\subsection{Proof of Lemma \ref{lem_oder1}}
\label{subsec_proof_lemma1}
In this section, we proceed with \cref{assump_sansH} or \cref{assump_apriori} in effect.
\subsubsection{Preliminaries and a Duhamel formula}

\begin{lemma}
    \label{lemma_domain}
    $\orho\in \xK^1_+\cap \xK_{\oLambda^2}$, then $\orho$ is in the domain of the Lindbladian and 
    \begin{align}
        \xL(\orho)=\oG \orho + \orho \oG^\dag +\sum_{j=1}^{\Ndissip} \oL_j \orho \oL_j^\dag 
    \end{align}
\end{lemma}
Before proving this Lemma, let us recall that in line with \cref{rmk:ambiguous}, $\oL_j\orho_{s} \oL_j^\dag\in \xK^1$ is a slight abuse of notation. Indeed, one should consider $\oL_j^\dag$  as an element of $B(\xH,\xH^{1,*})$ (the adjoint of the continuous operator $\oL_j: \xH^1\to \xH$) and not as $\oL_j^\dag: \D(\oL^\dag_j) \to \xH$ (the adjoint of the unbounded operator $\oL_j: \D(\oL_j) \to \xH$). Hence, with our convention $\oL_j\orho_{s} \oL_j^\dag$ is indeed a bounded operator on $\xH$ and an element of $\xK^1$. The same applies to $\orho \oG^\dag$.
\begin{proof}
    We use that the linear span $\mathcal{U}$ generated by
    $\ket{u}\bra{v}$ for $u,v\in \xH^2 \subset \D(\oG)$ is in the domain of $\xL$ by \cref{prop_coreL} (and is even a core in our case). We decompose
    $\orho=\oLambda^{-1}\osigma\oLambda^{-1}$ and $\osigma=\sum_{k=0}^\infty \mu_k \ket{u_k}\bra{u_k}$ with $u_k$ an orthonormal basis and $\mu_k\in L^1(\N)$. The sequence
    \begin{align}
        \osigma_{m}&=\sum_{k=0}^m \mu_k \ket{ u_k}\bra{u_k}
    \end{align}
    converges towards $\osigma$. Besides $\orho_{m}=\oLambda^{-1}\osigma_m \oLambda^{-1}$ belongs to $\mathcal{U}$ and converges in $\xK^1$ towards $\orho$. Then
    \begin{align*}
        \xL(\orho_{m})&=\sum_{k=0}^m \mu_k \left(\ket{\oG\oLambda^{-1} u_k}\bra{\oLambda^{-1}u_k} +\ket{\oLambda^{-1} u_k}\bra{\oG\oLambda^{-1}u_k}+\sum_{j=1}^{\Ndissip} \ket{\oL_j\oLambda^{-1} u_k}\bra{\oL_j\oLambda^{-1}u_k}\right)\\
    &=\oG\orho_{m}+\orho_{m}\oG^\dag +\sum_{j=1}^{\Ndissip} \oL_j \orho_{m} \oL_j^\dag
    \end{align*}
    We have $\oG,\oL_j \in B(\xH^2,\xH)$. Thus, $\xL(\orho_{m})$ converges in $\xK^1$ which shows that $\orho \in \D(\xL)$ and 
    \begin{align*}
        \xL(\orho_{m})\xrightarrow[m \to \infty ]{\xK^1} \oG\orho+\orho \oG^\dag +\sum_{j=1}^{\Ndissip} \oL_j \orho \oL_j^\dag=\xL(\orho)
    \end{align*}
\end{proof}
Then we can prove the following Duhamel formula:
\begin{lemma}
    \label{lemma_duhamel}
    Let $(\orho_s)_{0\leq s \leq t}\in L^\infty(0,t;\xK^1_+\cap \xK_{\oLambda^2})$ be the solution of the Lindblad equation, then for all $0 \leq s \leq t$, we have
\begin{align}
	\label{eq_duhamel}
    \orho_t&=e^{t\oG}\orho_0 e^{t\oG} + \sum_{j=1}^{\Ndissip}\int_{0}^{ t}  e^{(t-s)\oG}\oL_j\orho_{s} \oL_j^\dag e^{(t-s)\oG^\dag}  ds.
\end{align}
\end{lemma}

\begin{proof}
   The application $[0,t] \to \xK^1$ defined by $s\mapsto e^{(t-s)\oG} \orho_s e^{(t-s)\oG^\dag}$ is differentiable and
    
    \begin{align}
        \label{eq_equationon_tilderho}
        \frac{d}{ds}(e^{(t-s)\oG} \orho_s e^{(t-s)\oG^\dag}) &=
        e^{(t-s)\oG} \oG\orho_s e^{(t-s)\oG^\dag} + e^{(t-s)\oG} \orho_s\oG e^{(t-s)\oG^\dag }\\
        &+e^{(t-s)\oG} \xL(\orho_s) e^{(t-s)\oG^\dag}\\
        &=
        \sum_{j=1}^{\Ndissip} e^{(t-s)\oG} \oL_j \orho_s  \oL_j^\dag e^{(t-s)\oG^\dag}
    \end{align}
   Then, we integrate over $[0,t]$ to concludes.
\end{proof}

Let $\orho_0\in \xK_+^1\cap \xK_{\oLambda^4}$, by \cref{lem_apriori_estimate_sans_H}, for every $t\geq 0$, $\orho_t \in
\xK_{\oLambda^4}$, thus it belongs to $\xK_{\oLambda^2}$. Consequently, we ascertain the applicability of the Duhamel formula established in \cref{lemma_duhamel}. Thus, the proof of \cref{lem_oder1} is reduced to demonstrating the existence of positive constants $C_0$ and $C_j$ such that
\begin{align}
\label{eq_M0}
\|e^{\Delta t \oG}\orho_te^{\Delta t \oG^\dag} -\widetilde \oM_{0,\Delta t} \orho_t \widetilde \oM_{0,\Delta t} ^\dag\|_1                                                & \leq C_0 \|\orho\|_{L^\infty(0,T;\xK_{\oLambda^4})} \Delta t^2, \\
\label{eq_M1}
\|\int_{0}^{\Delta t} e^{(\Delta t-s)\oG}\oL_j \orho_{t+s} \oL_j^\dag e^{(\Delta t-s)\oG^\dag} ds -\widetilde \oM_{j,\Delta t} \orho_t \widetilde \oM_{j,\Delta t} ^\dag\|_1 & \leq C_j \|\orho\|_{L^\infty(0,T;\xK_{\oLambda^4})} \Delta t^2.
\end{align}
\subsubsection{Proof of estimate \eqref{eq_M0}}
\label{sec_proof_estimate_M0_sansH}
First, we have that $(e^{t\oG})_{t\geq 0}$ is a semigroup of contraction, thus $\|e^{t\oG}\|_\infty\leq 1$. Besides, $\widetilde \oM_{0,\Delta t}^\dag \widetilde \oM_{0,\Delta t} \leq \Id$, hence $\|\widetilde \oM_{0,\Delta t}\|_\infty \leq 1$.
Therefore, by introducing the cross-terms $$\pm (e^{\Delta t \oG}- \widetilde \oM_{0,\Delta t})\orho_t e^{\Delta t \oG^\dag},$$
we compute
\begin{align*}
& \|e^{\Delta t \oG}\orho_te^{\Delta t \oG} -\widetilde \oM_{0,\Delta t} \orho_t \widetilde \oM_{0,\Delta t} ^\dag\|_1                                                          \\
& \leq \| (e^{\Delta t \oG}-\widetilde \oM_{0,\Delta t}) \orho_t e^{\Delta t \oG}\|_1 + \| \widetilde \oM_{0,\Delta t} \orho_t (e^{\Delta t \oG}-\widetilde \oM_{0,\Delta t} ^\dag)\|_1 \\
& \leq 2 \| (e^{\Delta t \oG}-\widetilde \oM_{0,\Delta t}) \orho_t\|_1.
\end{align*}
All that remains to do is prove that $\| (e^{\Delta t \oG}-\widetilde \oM_{0,\Delta t})\|_{B(\xH^4,\xH)}\leq C (\Delta t)^2$.

\paragraph*{Under \cref{assump_sansH} ($\oH=0$)}
We have
$$\widetilde \oM_{0,\Delta t}= (\Id - \Delta t \frac{\oQ}{2})(\Id + (\Delta t \frac{\oQ}{2})^2)^{-1/2}.$$
Besides, for any positive real number $x\geq 0$,
$$0\leq e^{-x}-\frac{1-x}{\sqrt{1+x^2}}\leq x^2.$$
This implies that
\begin{align}
    0 \leq e^{-\Delta t \frac{\oQ}{2}}-\widetilde \oM_{0,\Delta t}\leq \frac{(\Delta t \oQ)^2}{4}.
\end{align}
As a consequence,
\begin{align}
    \|e^{-\Delta t \frac{\oQ}{2}}-\widetilde \oM_{0,\Delta t}\|_{B(\xH^4,\xH )}\leq C \Delta t^2,
\end{align}
leading to
\begin{align}
\| (e^{\Delta t \oG}-\widetilde \oM_{0,\Delta t}) \orho_t\|_1 & \leq C \Delta t^2 \| \orho_t \|_{\xK_{\oLambda^4}}\leq C \Delta t^2 \| \orho \|_{L^\infty(0,T;\xK_{\oLambda^4})}
\end{align}
Which concludes the proof of estimate \eqref{eq_M0}.

\paragraph*{Under \cref{assump_apriori} ($\oH\not =0$)}
We start with the following Taylor expansion (see e.g. \cite[Proposition 1.1.6]{butzerSemiGroupsOperatorsApproximation1967})
\begin{align}
    (e^{\Delta t \oG }- (\Id +\Delta t \oG))\ket{\psi}= \int_0^{\Delta t} (\Delta t-s)e^{s \oG } \oG^2 \ket{\psi}, \quad \ket{\psi} \in \D(\oG^2)
\end{align} 
for every $\ket{\psi}\in \xH^4 \subset \D(\oG^2)$. In the next equations, we consider both side of the equation as elements of $B(\xH^4,\xH)$. We can then compute 
\begin{align*}
    \oM_{0,\Delta t} - (\Id + \Delta t \oG)&=(\Id-i\frac{\Delta t}{2}  \oH)(\Id + i\frac{\Delta t}{2}  \oH)^{-1} (\Id- \frac{1}{2} \Delta t \oQ)- (\Id-i \Delta t \oH - \frac{\Delta t}{ 2}\oQ)\\
    &= \left( (\Id-i\frac{\Delta t}{2}  \oH)(\Id + i\frac{\Delta t}{2}  \oH)^{-1} - (\Id -\Delta t \oH) \right)(\Id- \frac{1}{2} \Delta t \oQ)\\
    &+ (\Id- i \Delta t \oH)(\Id- \frac{1}{2} \Delta t \oQ) - (\Id-i \Delta t \oH - \frac{\Delta t}{ 2}\oQ)\\
    &=\left( (\Id-i\frac{\Delta t}{2}  \oH)(\Id + i\frac{\Delta t}{2}  \oH)^{-1} - (\Id -i\Delta t \oH) \right)(\Id- \frac{1}{2} \Delta t \oQ) +i (\Delta t)^2 \oH\oQ.
\end{align*}
As $\oH\in B(\xH^2,\xH)$ and $\oQ\in B(\xH^4,\xH^2)$, we have $\oH\oQ \in B(\xH^4, \xH)$. Thus, we only have to tackle the first term of the last expression. We denote 
$$\oA_{\Delta t}=\left( (\Id-i\frac{\Delta t}{2}  \oH)(\Id + i\frac{\Delta t}{2}  \oH)^{-1} - (\Id -i\Delta t \oH) \right).$$
For every $x\in \R$, we have
\begin{align}
    \left| \frac{1-\frac{i}{2}x}{1+\frac{i}{2}x}-(1-ix)\right|^2\leq \min(x^2,x^4).
\end{align}
Thus,
\begin{align}
    \oA_{\Delta t}^\dag \oA_{\Delta t} \leq (\Delta t)^2 \oH^2,\quad \oA_{\Delta t}^\dag \oA_{\Delta t} \leq (\Delta t)^4 \oH^4
\end{align}
Multiplying by $\oLambda ^{-1}$ on both side of the first equations shows that $\|\oLambda^{-1}\oA_{\Delta t}^\dag \oA_{\Delta t}\oLambda^{-1}\|_\infty \leq C (\Delta t)^2$, implying $\|\oA_{\Delta t}\|_{B(\xH^2,\xH)}\leq C \Delta t$. Similarly we get $\|\oA_{\Delta t}\|_{B(\xH^4,\xH)}\leq C (\Delta t)^2$. Together with the fact that $\oQ\in B(\xH^4,\xH^2)$, we get
\begin{align}
    \|\oM_{0,\Delta t} - (\Id + \Delta t \oG)\|_{B(\xH^4,\xH)}\leq C (\Delta t)^2.
\end{align}
Thus,
\begin{align}
    \|(e^{\Delta t \oG }-\oM_{0,\Delta t})\orho_t\|_1 
    &\leq C \Delta t^2 \|\orho_t\|_{\xK_{\oLambda^4}}.
\end{align}
Next, we have 
\begin{align}
    (\oM_{0,\Delta t}-\widetilde \oM_{0,\Delta t})&=(\Id-i\frac{\Delta t}{2}  \oH)(\Id + i\frac{\Delta t}{2}  \oH)^{-1} (\Id- \frac{1}{2} \Delta t \oQ)(\Id-\oS_{\Delta t}^{-1/2}).
\end{align}
As $(\Id-i\frac{\Delta t}{2}  \oH)(\Id + i\frac{\Delta t}{2}  \oH)^{-1}$ is unitary, we get
\begin{align*}
    \|(\oM_{0,\Delta t}-\widetilde \oM_{0,\Delta t})\|_{B(\xH^4,\xH)}=\|(\Id- \frac{1}{2} \Delta t \oQ)(\Id-\oS_{\Delta t}^{-1/2})\|_{B(\xH^4,\xH)}.
\end{align*}
Then, using that for every positive real number $x \geq 0$,
\begin{align}
    -x^2 \leq (1-x)\left( 1 - \frac{1}{\sqrt{1+x^2}}\right) \leq x^2,
\end{align}
we obtain 
\begin{align}
    -(\Delta t)^2 \frac{\oQ^2}{4} \leq (\Id- \frac{1}{2} \Delta t \oQ)(\Id-\oS_{\Delta t}^{-1/2})\leq (\Delta t)^2 \frac{\oQ^2}{4}.
\end{align}
Hence, $\|(\Id- \frac{1}{2} \Delta t \oQ)(\Id-\oS_{\Delta t}^{-1/2})\|_{B(\xH^4,\xH)} \leq C (\Delta t)^2$, implying
\begin{align}
    \|(\oM_{0,\Delta t}-\widetilde \oM_{0,\Delta t})\orho_t\|_1
    &\leq C \Delta t^2 \|\orho_t\|_{\xK_{\oLambda^4}}
\end{align}
    \begin{comment}
    \paragraph{Second case: under the assumption $\sup_{t\in [0,T]}\xtr{(L^\dag L)^4\orho_t}< \infty .$}
    We start by applying Cauchy-Schwarz to one of the cross-term
    \begin{align}
        \label{eq_\oM_0_CS}
        \|(e^{-\frac{\Delta t}{2}L^\dag L}-\widetilde \oM_{0,\Delta t})\orho_t e^{-\frac{\Delta t}{2}L^\dag L}\|_1\leq  \| (e^{-\frac{\Delta t}{2}L^\dag L}-\widetilde \oM_{0,\Delta t}) \orho_t^{1/2}\|_2 \|\orho_t^{1/2}e^{-\frac{\Delta t}{2}L^\dag L}\|_2.
    \end{align}
    Using that $e^{-\frac{\Delta t}{2}L^\dag L}$ is bounded in operator norm by $1$, we get
    \begin{align}
        \|\orho_t^{1/2}e^{-\frac{\Delta t}{2}L^\dag L}\|_2^2= \sqrt{\xtr{e^{-\frac{\Delta t}{2}L^\dag L} \orho_t e^{-\frac{\Delta t}{2}L^\dag L}}}
        \leq 1
    \end{align}
    The first element in the product of the right-hand side of \cref{eq_\oM_0_CS} can be estimated as follows
    \begin{align}
        \| (e^{-\frac{\Delta t}{2}L^\dag L}-\widetilde \oM_{0,\Delta t}) \orho_t^{1/2}\|_2&= \sqrt{\xtr{\orho_t^{1/2}(e^{-\frac{\Delta t}{2}L^\dag L}-\widetilde \oM_{0,\Delta t})^2\orho_t^{1/2}}}\\
        & \leq \frac{\Delta t^2}{4}\sqrt{\xtr{(L^\dag L)^4 \orho_t}},
    \end{align}
    where we used that
    \begin{align*}
        0 \leq (e^{-\frac{\Delta t}{2}L^\dag L}-\widetilde \oM_{0,\Delta t})^2\leq \frac{(\Delta t L^\dag L)^4}{16}.
    \end{align*}
    The exact same reasoning with 
    $\|(e^{-\frac{\Delta t}{2}L^\dag L}-\widetilde \oM_{0,\Delta t})\orho_t \widetilde \oM_{0,\Delta t}\|_1$ finishes the proof of estimate \eqref{eq_M0}.
\end{comment}
\subsubsection{Proof of estimate \eqref{eq_M1}}
\label{sec_proof_estimat_eq_M1}
    The sketch of proof is the following:
\begin{align}
    \tag{Step 1}
    \int_0^{\Delta t} e^{(\Delta t-s)\oG}\oL_j \orho_{t+s}  \oL_j^\dag e^{(\Delta t-s)\oG^\dag}ds
    &=\int_0^{\Delta t}
e^{(\Delta t-s)\oG}\oL_j \orho_{t}  \oL_j^\dag e^{(\Delta t-s)\oG^\dag} ds +O(\Delta t^2)\\
\tag{Step 2}
&= \oM_{j,\Delta t} \orho_t \oM_{j,\Delta t}^\dag+O(\Delta t^2)\\
\tag{Step 3}
&=\widetilde \oM_{j,\Delta t} \orho_t \widetilde \oM_{j,\Delta t}^\dag+O(\Delta t^2),
\end{align}
where $O(\Delta t^2)$ means an element of $\xK^1$ with trace norm bounded by $C \Delta t ^2 \|\orho\|_{L^\infty(0,T;\xK_{\oLambda^4})}$ for $\Delta t$ small enough.
\paragraph*{Step 1}
\begin{align}
   \notag \|e^{(\Delta t-s)\oG}\oL_j (\orho_{t+s}-\orho_t) \oL_j^\dag e^{(\Delta t-s)\oG^\dag}\|_1 &\leq \|\oL_j (\orho_{t+s}-\orho_t) \oL_j^\dag\|_1\\
    \notag &\leq \int_0^s \|\oL_j \xL(\orho_{t+s'}) \oL_j^\dag\|_1 ds'\\
    &= \int_0^s \|\oL_j (\oG\orho_{t+s'}-\orho_{t+s'} \oG^\dag +\sum_{i=1}^{\Ndissip} \oL_i\orho_{t+s'}\oL_i^\dag  )\oL_j^\dag\|_1 ds'.
\end{align}

We recall that we assumed that $\oL_i \in B(\xH^5,\xH^4)$, besides we have already shown that $\oL_i \in B(\xH^1,\xH^0)$ thus by classical result of interpolation (see \cref{app_interpol}), $\oL_i \in B(\xH^{l+1},\xH^l)$ for $0\leq l\leq 4$.
As a consequence, we have $\oL_j\oL_i\in B(\xH^2,\xH)$, implying
\begin{align*}
    \|\oL_j\oL_i\orho_{t+s'}\oL_i^\dag \oL_j^\dag\|_1 \leq \|\oL_j\oL_i\|^2_{B(\xH^2,\xH)}\|\orho_{t+s'}\|_{\xK_{\oLambda^2}}.
\end{align*}
Similarly, $\oL_j\oG$ belongs to $B(\xH^3,\xH)$, thus
\begin{align}
    \|\oL_j \oG \orho_{t+s'} \oL_j^\dag\|_1 &\leq \|\oL_j \oG\|_{B(\xH^3,\xH)} \|\orho_{t+s'}\|_{\xK_{\oLambda^3}} \|\oL_j\|_{B(\xH^{1},\xH)}.
\end{align}
Leading to
\begin{align}
    \|e^{(\Delta t-s)\oG}\oL_j (\orho_{t+s}-\orho_t) \oL_j^\dag e^{(\Delta t-s)\oG^\dag}\|_1\leq s C \|\orho_{t+s'}\|_{\xK_{\oLambda^3}}.
\end{align}
We integrate for $s\in [0,\Delta t]$ and this concludes step 1
\paragraph*{Step 2}
We start with
\begin{align}
    \|e^{(\Delta t-s)\oG}\oL_j \orho_{t} \oL_j^\dag e^{(\Delta t-s)\oG^\dag}- \oL_j \orho_{t} \oL_j^\dag\|_1 &\leq 2 \|(e^{(\Delta t-s)\oG}- \Id)\oL_j \orho_{t} \oL_j^\dag\|_1
\end{align}
As $\|(e^{(\Delta t-s)\oG}- \Id)\|_{B(\xH^2,\xH)}\leq C (\Delta t-s)$, we get
\begin{align*}
    \|(e^{(\Delta t-s)\oG}- \Id)\oL_j \orho_{t} \oL_j^\dag\|_1\leq C (\Delta t-s) \|\oL_j\|_{B(\xH^3,\xH^2)} \|\orho_{t}\|_{\xK_{\oLambda^3}} \|\oL_j \|_{B(\xH^1,\xH)}
\end{align*}
Integration for $s\in[0,\Delta t]$ leads to the end of step 2.

\paragraph*{Step 3}
We have
\begin{align}
    \oM_{j,\Delta t}-\widetilde \oM_{j,\Delta t}= \sqrt{\Delta t} \oL_j (\Id-\oS_{\Delta t}^{-1/2})
\end{align}
Thus
\begin{align}
    0\leq (\oM_{j,\Delta t}-\widetilde \oM_{j,\Delta t})^\dag (\oM_{j,\Delta t}-\widetilde \oM_{j,\Delta t})&= \Delta t (\Id-\oS_{\Delta t}^{-1/2}) \oL_j^\dag \oL_j (\Id-\oS_{\Delta t}^{-1/2})\\
    & \leq \Delta t (\Id-\oS_{\Delta t}^{-1/2}) \oQ (\Id-\oS_{\Delta t}^{-1/2})
\end{align}
Besides, for every $x\geq 0$, a standard computation shows 
\begin{align}
    0 \leq 2x(1-\frac{1}{\sqrt{1+x^2}})^2\leq  x^4.
\end{align}
Thus,
\begin{align}
    \Delta t (\Id-\oS_{\Delta t}^{-1/2}) \oQ (\Id-\oS_{\Delta t}^{-1/2}) \leq \Delta t^4 \frac{\oQ^4}{2^4}.
\end{align}
As $\oQ^2 \in B(\xH^4,\xH)$, we obtain
\begin{align}
    \|\oM_{j,\Delta t}-\widetilde \oM_{j,\Delta t}\|_{B(\xH^4,\xH)}\leq C \Delta t^2,
\end{align}
which concludes the proof of estimate \eqref{eq_M1}, thus of \cref{lem_oder1} and \cref{th_onedissipator_firstorder}.

\section{Galerkin approximation, stability of the schemes and numerical results}
\label{sec_Galerkin_and_num}
Now that we have established the convergence of the quantum channel schemes defined in \cref{eq_def_scheme} for unbounded Lindblad equations, let us dive into the implications for the time discretization of a Galerkin approximation of the Lindblad equation.
\subsection{Galerkin approximation}
\label{ref_Galerkin}
 We consider an increasing sequence $(\xH_n)_n$ of Hilbert space included in $\xH$. We raise the attention of the reader to the fact that $\xH_n$ is not related to the space $\xH^n$ introduced in \cref{subsec_H_n}.
We denote by $\oP_n$ the orthogonal projectors onto $\xH_n$. We assume that $\oP_n$ converges strongly towards $\Id$ and that for all $n$, $\xH_n \subset \D(\oG)$. Then, one can define the truncated operators
\begin{align}
   \oH_n=\oP_n \oH \oP_n, \quad \oL_{j,n}=\oP_n \oL_{j} \oP_n
\end{align}
and consider the Lindbladian $\xL_n$ defined by
\begin{align}
    \label{eq-Lindblad_truncated}
    \xL_n(\orho)=-i[\oH_n,\orho] + \sum_{j=1}^{\Ndissip} D_{\oL_{j,n}}(\orho).
\end{align}
As the operators $\oH_n$ and $\oL_{j,n}$ are bounded, $\xL_n$ is the generator of a uniformly continuous quantum Markov semigroup as studied initially by Lindblad in \cite{lindbladGeneratorsQuantumDynamical1976}.
\begin{example}
    \label{ex-fock}
	In the case $\xH = L^2(\R,\C)$, we will consider the typical choice $\xH_n=\{ \ket{k}\mid 0 \leq k\leq n\}$, where $(\ket{n})_{n \in \xN}$ denotes the so-called Fock basis (that is the eigenbasis of a quantum harmonic oscillator).
\end{example}
%\begin{remark}
%    Note that the Lindbladian introduce in \cref{eq-Lindblad_truncated} is not the only choice to project the dynamics on density matrix of the finite dimensional Hilbert space $\xH_N$. If $V,W$ are operarors on $\xH$, $V_NW_N \not = (VW)_N$ in general, thus one has to be careful as for e.g. the discretization
%    $$\Gamma^\dag_N \orho \Gamma_N -\frac{1}{2}( P_N\Gamma^\dag\Gamma P_N \orho+\orho P_N\Gamma^\dag\Gamma P_N )$$ is not a dissipator in general, hence do not lead to a Lindbladian.\\
%    Note also that if one approxime $\Gamma$ on $\xH_N$ by a product of truncated operaror, like for example $a_N a^\dag_N$, one does not always fit in our framework (e.g. \cref{eq-L-L_N_projected} does not hold). Nevertheless, $a^\dag_Na_N + \Id_N=(a^\dag \oa + \Id)_N$ fits.
%\end{remark}

%Let us now consider a density operaror $\orho_0$ and define $\orho(t)$ the solution of \cref{eq-lindblad_full} with initial condition $\orho_0$.

\subsection{Explicit analysis of time discretization of the multi-photons loss channel}
We consider the following Lindblad equation on $\mathcal H = L^2(\R,\C)$:
\begin{align}
    \label{eq_photon_loss}
    \frac{d}{dt} \orho =D[\oa^l](\orho),
\end{align}
where $\oa$ is the usual annihilation operator on $\mathcal H$ and $l\geq 1$ is a fixed parameter;
this equation models, for instance, a dissipative process where a quantum harmonic oscillator loses $l$-uplets of photons. For $l=1$ we obtain the so-called photon loss channel which is the dominant error channel in many quantum experiments, for example in quantum optics or superconducting circuits. On the other hand, for $l\geq 2$, the coherent loss of $l$-uplets of photons can be used as a powerful resource for bosonic error correction--see e.g. \cite{mirrahimiDynamicallyProtectedCatqubits2014,leghtasConfiningStateLight2015}.

It is known that the solutions of \cref{eq_photon_loss} converge, as $t$ goes to infinity, to a density operator that depends on the initial condition and has support on $\operatorname*{Span}\left\{\ket{j}\mid 0\leq j \leq l-1\right\}$.
This is, for instance, a direct application of the results in \cite{azouitWellposednessConvergenceLindblad2016}.

The aim of this section is to show that an explicit Euler scheme is subject to a Courant--Friedrichs--Lewy (CFL) type condition: the dimension of the discretization space $\xH_n$ imposes a constraint on the smallness of the time-step $\Delta t$ to ensure the stability of the scheme. On the other hand, the quantum channel schemes introduced in this article do not suffer from this weakness and exhibit a robustness in the time discretization independent of the Galerkin truncation.
\begin{remark}
    \label{rmk_stability_photon_loss}

    Note that, for any $n\in \N$, the fact that $\xH_n$ is stable by $\oa$ entails that the convex cone $\xK_+^1 \cap B(\xH_n)$ is stable under \cref{eq_photon_loss}:
	in fact, we even have that if $\orho_0\in \xK_+^1 \cap B(\xH_n)$, then $\S_t \orho_0=e^{t\xL_n}\orho_0$.
\end{remark}

\subsubsection{A CFL condition for the first-order explicit Euler scheme}
\label{subsubsec:CFL}
The first-order explicit Euler scheme applied to the Galerkin truncation $n$ yields:
\begin{align}
    \mathcal{E}_1^{\Delta t,n}(\orho)= \orho + \Delta t \left( \oa^l_n \orho \oa^{l\,\dag}_n -\frac{1}{2}(\oa^{l\,\dag}_n \oa^l_n \orho + \orho \oa^{l\,\dag}_n \oa^l_n)\right).
\end{align}
A simple computation shows that for any $0\leq  k \leq n$, we have
    \begin{align}
        \oa^{l\,\dag}_n \oa^l_n \ket{k}=c_k \ket{k},
	    \quad \oa^l_n \ket k = \sqrt{c_k} \ket{k-l},
	    \quad c_{k}=\Pi_{j=0}^{l-1}(k-j).
    \end{align}
Note that we don't make the dependence of $c_k$ on the parameter $l$ explicit to alleviate the notations. 
\begin{lemma}[CFL condition on $l$-photon loss.] If $\Delta t c_n > 2$, the scheme $\mathcal{E}_1^{\Delta t,n}$ is unstable.
\end{lemma}
\begin{proof}
Focusing on the $(n,n)$ term of $\orho$, we have the decoupled equation
    \begin{align}
	    \bra{n}\mathcal{E}_1^{\Delta t,n}(\orho)\ket{n}= (1- \Delta t c_n) \bra{n}\orho\ket{n}
    \end{align}
This shows the unstability of the scheme if $\Delta t c_n>2$.
\end{proof}
Conversely, we have the following stability property for $\Delta t c_n < 2$:
\begin{lemma}
    For a given $n\in \xN$, if $\Delta t c_n < 2$, then for any initial condition $\orho_0 \in \xH_n$, there exists $\orho_\infty \in B(\xH_n)$ such that $P_l\orho_\infty P_l=\orho_\infty$ and
    $(\mathcal{E}_1^{\Delta t,n})^p(\orho)\xrightarrow[p \to \infty]{} \orho_\infty$ at a geometric rate.
\end{lemma}
\begin{proof}
    Let us assume $\Delta t c_n < 2$, we have the system of coupled equations
    \begin{align}
        \label{eq_multiphotons_scheme}
        \bra{k}\mathcal{E}_1^{\Delta t,n}(\orho)\ket{m}= \left\{
        \begin{array}{cr} 
		\left( 1- \Delta t \frac{c_k+c_m}{2}\right) \bra{k}\orho\ket{m}& \text{ if }\max \{k+l,m+l\}>n \\
		\\
		\left(1- \Delta t \frac{c_k+c_m}{2}\right) \bra{k}\orho\ket{m}+\Delta t \sqrt{c_{k+l} c_{m+l}} \bra{k+l}\orho\ket{m+l}& \text{ otherwise}
        \end{array}
        \right.
    \end{align}
Note that this system is triangular, \emph{i.e.} $\bra{k}\mathcal{E}_1^{\Delta t,n}(\orho)\ket{m}$
	depends only on values $\bra{k'} \orho \ket{m'}$ with $k'+m' \geq k+m$ --
	hence, we readily deduce that the eigenvalues of $\mathcal{E}_1^{\Delta t,n}$ are given
	by the diagonal coefficients $\left(1- \Delta t \frac{c_k+c_m}{2}\right)$,
	all of magnitude less than $1$ since $0 \leq c_k \leq c_n$ for $k\leq n$.
 
%__%	Hence, the equations for $n-l < \max \{k,m\} \leq n$ are decoupled. Besides as $c_k \leq c_n$ for $k \leq n$, if $\max\{k,m\} \geq l$, the sequence $(\bra{k}(\mathcal{E}_1^{\Delta t,n})^p(\orho)\ket{m})_{p \in \N}$ converges geometrically towards $0$ with rate $1-\Delta t \frac{c_k+c_m}{2}$ that satisfies $0<\Delta t \frac{c_k+c_m}{2}\leq \Delta t c_n <2$.
%__%
%__%    Moreover, for $k,m \leq n$ and $\max\{k,m\} \geq l$, we have
%__%    \begin{align*}
%__%        \bra{k}\mathcal{E}_1^{\Delta t,n}(\orho)\ket{m}=\bra{k}\orho\ket{m}- \Delta t \frac{c_k+c_m}{2} \bra{k}\orho\ket{m}+\Delta t \sqrt{c_{k+l} c_{m+l}} \bra{k+l}\orho\ket{m+l}.
%__%    \end{align*}
%__%    Doing a proof by induction, we can assume that $\bra{k+l}(\mathcal{E}_1^{\Delta t,n})^p(\orho)\ket{m+l}$ converges geometrically towards $0$, and we deduce that $\bra{k}(\mathcal{E}_1^{\Delta t,n})^p(\orho)\ket{m}$ also converges geometrically towards $0$.
%__%
%__%    Finally, if $0\leq k,m \leq l$, then $c_k=c_l=0$, thus \cref{eq_multiphotons_scheme} becomes
%__%    \begin{align*}
%__%        \bra{k}\mathcal{E}_1^{\Delta t,n}(\orho)\ket{m}= \bra{k}\orho\ket{m}+\Delta t \sqrt{c_{k+l} c_{m+l}} \bra{k+l}\orho\ket{m+l}
%__%    \end{align*}
%__%	As $(\bra{k+l} \left( \mathcal{E}_1^{\Delta t,n}\right)^p(\orho) \ket{m+l})_{p\in \N}$ converges geometrically towards $0$, $\bra{k}\mathcal{E}_1^{\Delta t,n}(\orho)\ket{m}$ has a limit and converges geometrically towards it.
%__%
\end{proof}
\subsubsection{Proof of convergence of the quantum channel schemes}
\label{subsec:cvg_ex_cptp}
Let us now show that the $l$-photon loss channel, modeled by \cref{eq_photon_loss}, satisfies \cref{assump_sansH} with $\oLambda=\Id + \oa^{l\dag}\oa^l$. Indeed, as $\oLambda^k$ and $\Id + (\oa^{\dag}\oa)^{kl}$ are relatively bounded with respect to one another, we have
\begin{equation}
    \xH^k=\{ \ket{\psi}=\sum_n \psi_n \ket{n}\mid \sum_n (1+n^{kl}) |\psi_n|^2 < \infty\}.
\end{equation}
We easily see that $\xH^2$ is dense in $\xH$, and $\oa^{l\dag}\oa^l$ with domain $\xH^2$ is self-adjoint. Besides $\oL=\oa^l$ is bounded from $\xH^{n+1}$ to  $\xH^{n}$ for all $n\in \N$. It remains to establish the \textit{a priori} estimate of \cref{apriori_sans_H}. Formally, we have
\begin{align*}
    \xL^*((\Id+ \oL^\dag \oL)^4)&=\oL^\dag (\Id+ \oL^\dag \oL)^4 \oL -\oL^\dag \oL (\Id+  \oL^\dag \oL)^4\\
    &=\oL^\dag [(\Id+ \oL^\dag \oL)^4,\oL].
\end{align*}
Introducing $f(k)=\prod_{j=0}^{l-1}(k-j)$, one can check that $\oL^\dag \oL=f(\oa^\dag \oa)$. Introducing now $g=(1+f)^4$ and noting that both $f$ and $g$ are increasing, we get
\begin{align}
    \oL^\dag [(\Id+ \oL^\dag \oL)^4,\oL]= \oa^{\dag l} [g(\oa^\dag \oa),\oa^l]=\oa^{\dag l} \oa^l (g(\oa^\dag \oa-l\Id )-g(\oa^\dag \oa))
    =f(\oa^\dag \oa) (g(\oa^\dag \oa-l\Id )-g(\oa^\dag \oa)) \leq 0
\end{align}
As a consequence, \cref{apriori_sans_H} holds with $\gamma =0$. We conclude that \cref{th_onedissipator_firstorder} holds; that is, for any $T>0$, there exists $C>0$ such that the scheme
\begin{align}
    \widetilde \oM_{0,\Delta t}&= (\Id - \frac{\Delta t}{2} \oL^\dag \oL)(\Id + \frac{(\Delta t \oL^\dag \oL)^2}{4})^{-1/2},\\
    \widetilde \oM_{1,\Delta t}&= \sqrt{\Delta t }\oL (\Id + \frac{(\Delta t \oL^\dag \oL)^2}{4})^{-1/2},\\
    \xF_{\Delta t} (\orho)&= \widetilde \oM_{0,\Delta t} \orho \widetilde \oM_{0,\Delta t}^\dag +\widetilde \oM_{1,\Delta t} \orho \widetilde \oM_{1,\Delta t}^\dag,
\end{align}
satisfies, for every $(\Delta t,N)\in \mathbb{R}^+ \times \mathbb{N}$ such that $\Delta t N = T$, and for every $\orho_0 \in \mathcal{K}^1_+ \cap \mathcal{K}_{\oLambda^4}$,
\begin{align}
    \|\orho_T-\xF_{\Delta t}^N(\orho_0)\|_1 \leq C \|\orho_0\|_{\xK_{\oLambda^4}} \Delta t.
\end{align}

Introducing the same quantum channel scheme $\xF_{\Delta t,n}$ for the Lindblad equation $\xL_n$ on the Hilbert space $\xH_n$, we have
\begin{align}
    \widetilde \oM_{0,\Delta t,n}&= (\Id - \frac{\Delta t}{2} \oL_n^\dag \oL_n)(\Id + \frac{(\Delta t \oL_n^\dag \oL_n)^2}{4})^{-1/2},\\
    \widetilde \oM_{1,\Delta t,n}&= \sqrt{\Delta t }\oL_n (\Id + \frac{(\Delta t \oL_n^\dag \oL_n)^2}{4})^{-1/2},\\
    \xF_{\Delta t,n} (\orho)&= \widetilde \oM_{0,\Delta t,n} \orho \widetilde \oM_{0,\Delta t,n}^\dag +\widetilde \oM_{1,\Delta t,n} \orho \widetilde \oM_{1,\Delta t,n}^\dag.
\end{align}
Note that the restriction of $\xF_{\Delta t}$ to element of $B(\xH_n)$ coincide with $\xF_{\Delta t,n}$,
since $\xH_n$ is stable by both $\oL$ and $\oL^\dag \oL$. This leads to:
\begin{proposition}
    \label{prop_stability_scheme_multi_photon}
    Let $l \in \N$ and $T>0$. There exists a positive constant $C>0$ such that, for every $n \in \N$, every $\orho_0 \in \xK^1_+ \cap B(\xH_n)$ and every $(\Delta t,N)\in \R^+ \times \xN$ such that $\Delta t N =T$, we have 
    \begin{align}
	    \|\S_t(\orho_0)-\xF_{\Delta t,n}^N(\orho_0)\|_1 \leq C \|\orho_0\|_{\xK_{\oLambda^4}} \Delta t.
    \end{align}
\end{proposition}
As mentionned above in \cref{rmk_stability_photon_loss}, we have $\S_t (\orho_0)=e^{t \xL_n}\orho_0$ when $\orho_0 \in \xH_n$.
Besides,
$\orho_0\in \xH_n$, i.e., $\orho_0$ is supported only on the first $n+1$ Fock states
and in particular
$\orho_0 \in \xK_{\oLambda^4}$.

We emphasize once again that the fact that $\S_t$ and $e^{t\xL_n}$ coincide on $B(\xH_n)$ is a very specific feature of the multi-photon loss channel, and is not to be generically expected in other contexts.
Thus, the generalization of \cref{prop_stability_scheme_multi_photon} to other dynamics would require in particular errors estimates on $\|(\S_t-e^{t\xL_n})\orho_0\|_1$ and is left to future work. Nevertheless, in the next section, we numerically illustrate how the quantum channel schemes perform in some other examples.
\subsection{Numerical implementation and performance comparisons}
\begin{figure}
    \centering
    \includegraphics[width=\textwidth]{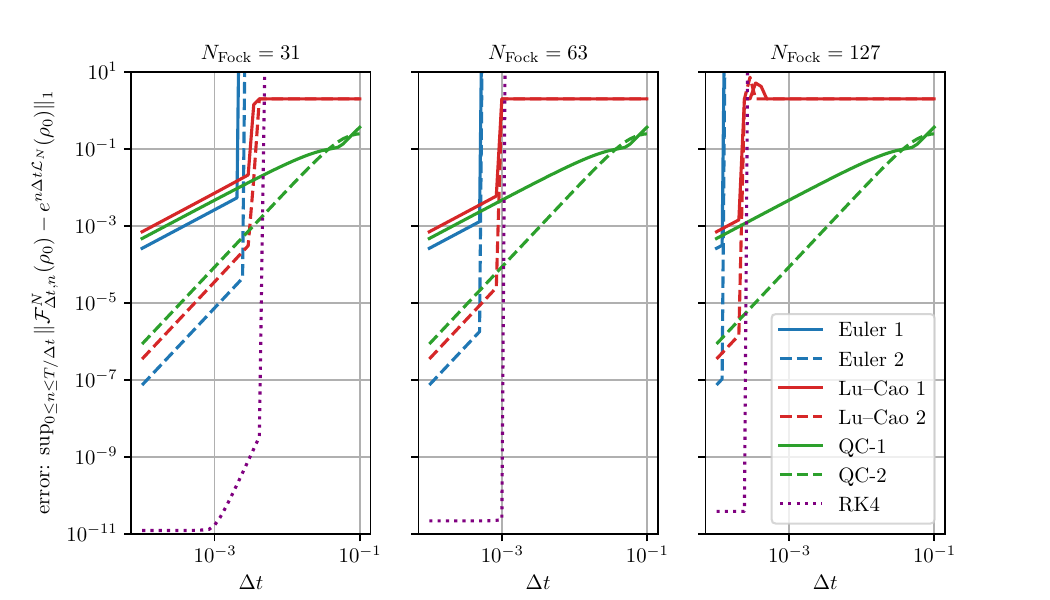}
    \includegraphics[width=\textwidth]{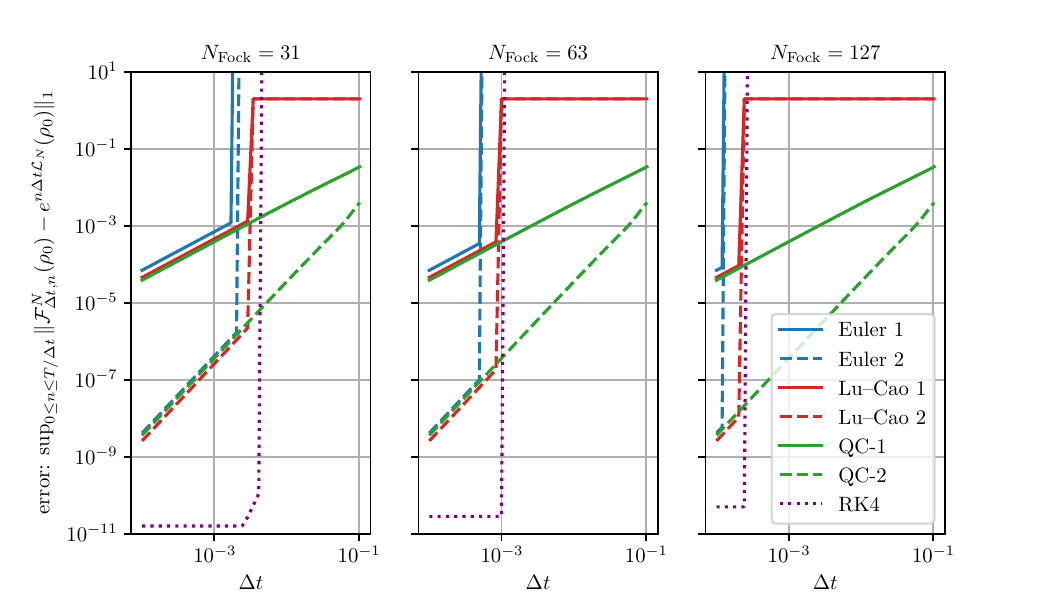}
    \caption{
	    Comparison of the Euler, Lu--Cao, Quantum channel and Runge--Kutta solvers. These methods are benchmarked on simulations of the preparation of a dissipative cat qubit (see \cref{eq_inflation_cat}) in the top row and a Z-Gate on a dissipative cat qubit (see \cref{eq_z_gate_cat}) on the bottom row, for increasing dimension $N_{\mathrm{Fock}}$ of the truncated Hilbert space.
	For each case, a reference solution is computed using an adaptive Dormand--Prince method with relative and absolute precisions set at $10^{-14}$.
	The error of each solver is then approximated as the supremum in time of the trace norm distance between the result of the solver and this reference solution. For each plot, we computed the solution for 40 different values of $\Delta t$ and then connected these points with a continuous line as a guide for the eye.
	The saturation of the error of the Lu--Cao solvers comes from the fact that their error is structurally bounded: indeed, these schemes preserve the convex set of density operators, whose diameter is 2.
	}
    \label{fig:simus}
\end{figure}

In this section, we analyze the precision of the following schemes as a function of the Galerkin truncation size.
\begin{enumerate}
    \item Euler 1: the first-order explicit Euler scheme,
    \begin{align}
        \orho_{n+1}=\orho_{n}+\Delta t \xL(\orho_{n})
    \end{align}
    \item Euler 2: the second-order explicit scheme obtained from a second-order Taylor expansion
    \begin{align}
        \orho_{n+1}=\orho_{n}+\Delta t \xL(\orho_{n})+\frac{(\Delta t)^2}{2}\xL(\xL(\orho_{n}))
    \end{align}
    \item Lu--Cao 1: the first-order non-linear structure preserving\footnote{\label{ftnote_struct}Structure preserving means here that the set of density operator is preserved, not that the evolution is a CPTP map} scheme introduced in \cite{caoStructurePreservingNumericalSchemes2024} and recalled in \cref{subsec:Non-linear positivity preserving schemes}:
    \begin{align*}
        \xF_{\Delta t}(\orho)&= \frac{\sum_{j=0}^{\Ndissip} \oM_{j} \orho \oM_{j}^\dag}{\xtr{\sum_{j=0}^{\Ndissip} \oM_{j} \orho \oM_{j}^\dag}}\\
        \oM_j&= \begin{cases}
            \Id + \Delta t \oG, & \text{for } j=0\\
            \sqrt{\Delta t} \oL_j, & \text{for }  1\leq j \leq \Ndissip,
        \end{cases}
    \end{align*}
    \item Lu--Cao 2: a second order non-linear structure preserving\footnote{See \cref{ftnote_struct}} scheme introduced in \cite{caoStructurePreservingNumericalSchemes2024}:
    \begin{align*}
        \xF_{\Delta t}(\orho)&= \frac{\sum_{j=0}^{\Ndissip(\Ndissip+1)} \oM_{j} \orho \oM_{j}^\dag}{\xtr{\sum_{j=0}^{\Ndissip(\Ndissip+1)} \oM_{j} \orho \oM_{j}^\dag}}\\
        \oM_j&= \begin{cases}
            \Id + \Delta t \oG + \frac{\Delta t}{2}\oG^2, & \text{for } j=0\\
            \sqrt{\frac{\Delta t}{2}} \oL_j (\Id + \Delta t \oG), & \text{for }  1\leq j \leq \Ndissip\\
            \sqrt{\frac{\Delta t^2}{2}} \oL_{j \!\!\! \mod \Ndissip} \oL_{\floor{\frac{j}{\Ndissip}}}, & \text{for }  \Ndissip+1 \leq j \leq \Ndissip(\Ndissip+1),
        \end{cases}
    \end{align*}
    \item QC-1: the first order quantum channel scheme studied in this paper:
    \begin{align*}
        \xF_{\Delta t}(\orho)&= \sum_{j=0}^{\Ndissip+1} \widetilde \oM_{j} \orho \widetilde \oM_{j}^\dag\\
        \oM_j&= \begin{cases}
            (\Id-i\frac{\Delta t}{2}  \oH)(\Id + i\frac{\Delta t}{2}  \oH)^{-1} (\Id- \frac{\Delta t}{2}  \sum_{j=1}^{\Ndissip} \oL_j^\dag \oL_j), & \text{for } j=0\\
            \sqrt{\Delta t} \oL_j , & \text{for }  1\leq j \leq \Ndissip\\
        \end{cases}\\
        \oS&= \sum_{j} \oM_{j}^\dag \oM_{j}, \quad \widetilde \oM_j= \oM_j \oS^{-1/2}
    \end{align*}
    \item QC-2: a second order linear structure preserving scheme obtained by applying the linear normalization method presented in \cref{subsec_CPTP_schemes} to Lu--Cao 2, yielding
    \begin{align*}
        \xF_{\Delta t}(\orho)&= \sum_{j=0}^{\Ndissip(\Ndissip+1)} \widetilde \oM_{j} \orho \widetilde \oM_{j}^\dag\\
        \oM_j&= \begin{cases}
            \Id + \Delta t \oG + \frac{\Delta t^2}{2}\oG^2, & \text{for } j=0\\
            \sqrt{\frac{\Delta t}{2}} \oL_j (\Id + \Delta t \oG), & \text{for }  1\leq j \leq \Ndissip\\
            \sqrt{\frac{\Delta t^2}{2}} \oL_{j \!\!\! \mod \Ndissip} \oL_{\floor{\frac{j}{\Ndissip}}}, & \text{for }  \Ndissip+1 \leq j \leq \Ndissip(\Ndissip+1),
        \end{cases}\\
        \oS&= \sum_{j} \oM_{j}^\dag \oM_{j}, \quad \widetilde \oM_j= \oM_j \oS^{-1/2}
    \end{align*}
    \item Runge--Kutta 4, which is known not to preserve the structure of Lindblad equation \cite{rieschAnalyzingPositivityPreservation2019},
	    but remains one of the most commonly used black-box solvers in open quantum systems numerical librairies.
\end{enumerate}
\begin{table}
	\centering
\begin{tabular}{|c|c|c|c|c|c|}
    \hline
    Method& Linear & Preserves the set of density matrices & CPTP & Order \\
    \hline
    Euler 1 & \checkmark & \text{\sffamily X}   &   \text{\sffamily X}   & 1 \\
    \hline
    Euler 2 & \checkmark & \text{\sffamily X}   & \text{\sffamily X}  & 2 \\
    \hline
    Lu--Cao 1 & \text{\sffamily X}   & \checkmark & \text{\sffamily X} & 1 \\
    \hline
    Lu--Cao 2 & \text{\sffamily X}   & \checkmark & \text{\sffamily X} & 2 \\
    \hline
    QC-1 & \checkmark & \checkmark & \checkmark&   1 \\
    \hline
    QC-2 & \checkmark & \checkmark & \checkmark& 2 \\
    \hline
    Runge--Kutta 4 & \checkmark & \text{\sffamily X} & \text{\sffamily X}&  4 \\
    \hline
    \end{tabular}
    \caption{Qualitative properties of the numerical schemes.}
	\label{table-prop_scheme}
\end{table}

Note that, among those schemes, only the Lu--Cao and QC families are structure-preserving; the qualitative properties of the schemes are summarized in \cref{table-prop_scheme}.
We include the Euler and Runge--Kutta schemes as points of comparison of the expected performance when
one resorts to the strategy of ignoring the structure of the problem and relying merely on the convergence of the scheme to ensure that relevant properties are satisfied, in a black-box manner.

In all simulations presented below, the error is defined relative to a reference solution using a Dormand--Prince method with relative and absolute precisions set at $10^{-14}$.

We consider two example quantum dynamics, which were introduced in \cite{mirrahimiDynamicallyProtectedCatqubits2014} in for the modelization of so-called dissipative cat qubits.
The first example models the initialization of a dissipative cat qubit and corresponds to the following Lindblad equation on $\xH = L^2(\R,\C)$:
\begin{align}
    \label{eq_inflation_cat}
    \frac{d}{dt} \orho_t = D[\oa^2 - \alpha^2 \Id](\orho_t),\quad \orho_0=\ket{0}\bra{0}
\end{align}
where $\oa$ is the usual annihilation operator and $\alpha\in\C$ is a parameter. In this setting, the relevant quantity to compute is the steady state of the evolution.
We set $\alpha=2$ (representative of the typical range in existing experiments, for instance $\alpha \in [\sqrt{2.5},\sqrt{11.3}]$ in \cite{regladeQuantumControlCat2024})
and fix the final time of the simulation to $T=1$. The results for the simulations using a truncated Fock basis with $32,\, 64$ and $128$ states are reproduced in the top of \cref{fig:simus}.

The second example models a Z-gate on a dissipative cat qubit, following \cite{mirrahimiDynamicallyProtectedCatqubits2014}. The dynamics reads
\begin{align}
    \label{eq_z_gate_cat}
	\frac{d}{dt} \orho_t = -i [\epsilon_z (\oa + \oa^\dag) ,\orho_t] +\kappa_2 D[\oa^2 - \alpha^2\Id](\orho_t) +\kappa_1 D[\oa](\orho_t), \quad \orho_0=\ket{\mathcal C_\alpha^+} \bra{\mathcal C_\alpha^+}
\end{align}
where we define the cat state $\ket{\mathcal C_\alpha^+} = (\ket\alpha + \ket{-\alpha})/||\ket\alpha + \ket{-\alpha}||$ with the usual definition of coherent states $\ket \alpha = e^{-|\alpha|^2/2} \sum_{n\in\N} \alpha^n/\sqrt{n!} \ket n$.
In this case, the relevant quantity to compute is the state reached after a time $T=\pi/(4\alpha\epsilon_z)$, which would coincide with the application of a logical Z gate in the absence of the noise term proportional to $\kappa_1$ in the equation (modeling photon loss) and in the limit of large $\kappa_2/\epsilon_z$.
We pick the following parameters: $\alpha=2$, $\epsilon_z=0.2$, $\kappa_1=0.01$ and $\kappa_2=1$.
The results for the simulations using a truncated Fock basis with $32,\, 64$ and $128$ states are reproduced in the bottom of \cref{fig:simus}.

In both examples, we notably observe that QC-1 and QC-2 are the only accurate schemes for large values of $\Delta t$.
	The explicit Euler solvers as well as the Runge--Kutta 4 solver are markedly subject to a CFL condition, producing simulations that are highly inaccurate or even crash for $\Delta t$ not small enough. Note that the stability condition obtained in \cref{subsubsec:CFL} for the two-photon loss channel
	(\emph{i.e.} using $l=2$ in \cref{eq_photon_loss}),
	namely $\Delta t N(N-1)<2$,
	predicts stability for $\Delta t \leq 2.2 \times 10^{-3}$
	(respectively, $5.1 \times 10^{-4}$ and $1.2 \times 10^{-4}$)
	for $N=31$ (respectively, $63$ and $127$),
	in good agreement with the numerically observed boundary.
	Furthermore, notice that the Runge--Kutta 4 solver exhibits a very stiff change in behaviour,
	from being unstable and raising numerical errors for large $\Delta t$ to delivering extremely accurate solutions for $\Delta t$ below the CFL.
	Note also that, outside of this stability zone, which becomes smaller and smaller as one increases the size $N_{\mathrm{Fock}}$ of the Galerkin discretization, the non-linear schemes Lu--Cao are also largely inaccurate. On the other hand, the schemes QC-1 and QC-2, that preserve the structure of the quantum evolutions, exhibit a strong robustness and an error scaling corresponding to their expected order even for large time-steps.

Regarding execution time, our scheme (QC-1) is comparable to Euler 1 or Lu--Cao 1.
More precisely,
the operators $(\widetilde \oM_j)_{0 \leq j \leq \Ndissip}$
are computed once at the beginning of the simulation\footnote{This requires inverting and taking the square root of $\oS$, which is a matrix of size $(\dim \H_N)\times (\dim \H_N)$. In contrast, a vectorized representation of $\xL_N$ requires a matrix of size $(\dim \H_N)^2\times (\dim \H_N)^2$.}, and each time-step then simply consists in the application $\orho \mapsto \sum_j \widetilde \oM_j \orho \widetilde \oM_j^\dag$.
We report in \cref{table-cost} the exact number of matrix multiplications and additions required to compute one step of each numerical scheme.
\begin{table}[h!]
	\centering
    \begin{tabular}{| c | c |c| c| }
        \hline
	    Method & Order  & Matrix multiplications & Matrix additions \\
     \hline
	    (Application of the Lindbladian) & &  $2 \Ndissip+2$ & $ \Ndissip+2$ \\
	    &&&\\
	    \multirow{2}{3em}{Euler} &  1 & $2 \Ndissip+2$ & $ \Ndissip+3$  \\
	    &  2 & $4 \Ndissip+4$ & $2 \Ndissip+6$  \\
	    &&&\\
	    \multirow{2}{5em}{Lu--Cao} &  1 & $2 \Ndissip+2$ & $\Ndissip$  \\
	    &  2 & $2 \Ndissip^2 +2\Ndissip+2$ & $2 \Ndissip^2 +2\Ndissip$  \\
	    &&&\\
	    \multirow{2}{3em}{QC} &  1 & $2 \Ndissip+2$ & $\Ndissip$  \\
	    &  2 & $2 \Ndissip^2 +2\Ndissip+2$ & $2 \Ndissip^2 +2\Ndissip$  \\
	    &&&\\
	    Runge--Kutta & 4 & $8 \Ndissip+8$ & $4 \Ndissip+14$\\
     \hline
    \end{tabular}
	    \caption{Computational cost of a single time-step for the numerical schemes studied in this section, as a function of the number $N_d$ of dissipators entering the Lindblad equation. The cost of evaluation of the Lindbladian superoperator is reported in the first line for comparison.}
	\label{table-cost}
\end{table}

Our final benchmark compares the computational time required to achieve specific precision levels for the Z-gate on a dissipative cat qubit, as described in \cref{eq_z_gate_cat} and illustrated in the lower plot of \cref{fig:simus}. We set three desired precision levels: $10^{-2}$, $10^{-4}$, and $10^{-6}$. Using the lower plot of \cref{fig:simus}, we determine the $\Delta t$ needed to achieve each precision level. We then simulate each scheme on our laptop\footnote{Equipped with an Intel CPU i7-9850H with 12 cores.} with a simple python implementation and report both the chosen $\Delta t$ and the total time required for the simulation at each precision level in \cref{table-cost_time}. For better reproducibility, the reported simulation time is average over 10 repetitions.

Except for high precision and small $N_{\mathrm{Fock}}$, where the CFL condition does not constrain classical schemes, quantum channel schemes outperform the competition in terms of the required number of steps and execution time. Note also that for the quantum channel schemes, the largest required time-step for a given precision is independent of the size of Hilbert space $\xH_N$. 
\renewcommand{\arraystretch}{1.3}
\begin{table}[h!]
	\centering
\begin{tabular}{| c | c | c|c|c | c|c|c|c|c|c| }
        \hline
	    \multirow{2}{*}{$N_{\mathrm{Fock}}$} &\multirow{2}{*}{Method}   & \multicolumn{3}{c|}{Precision $10^{-2}$} & \multicolumn{3}{c|}{Precision $10^{-4}$} &\multicolumn{3}{c|}{Precision $10^{-6}$}  \\
          \cline{3-11}
          &  & exec. time & $N_{\mathrm{step}}$& $\Delta t$ & exec. time & $N_{\mathrm{step}}$& $\Delta t$ & exec. time& $N_{\mathrm{step}}$& $\Delta t$\\
     \hline
\multirow{7}{*}{31} & Euler 1 & 23 ms & 1154 & 1.7e-3 & 0.32 s & 16447 & 1.2e-4 & N.A. & 1391653 & (\textit{1.4e-6})\\
 & Lu--Cao 1 & 62 ms & 678 & 2.9e-3 & 0.88 s & 9668 & 2.0e-4 & N.A. & 907610 & (\textit{2.2e-6})\\
 & QC-1 & 8.3 ms & 80 & 2.5e-2 & 0.62 s & 8098 & 2.4e-4 & N.A. & 768994 & (\textit{2.6e-6})\\
\cdashline{2-2}
 & Euler 2 & 39 ms & 966 & 2.0e-3 & 38 ms & 966 & 2.0e-3 & \textbf{ 55 ms } & 1377 & 1.4e-3\\
 & Lu--Cao 2 & 0.15 s & 678 & 2.9e-3 & 0.15 s & 678 & 2.9e-3 & 0.25 s & 1154 & 1.7e-3\\
 & QC-2 & \textbf{ 8.1 ms } & \textbf{ 19 } & 1.0e-1 & \textbf{ 36 ms } & \textbf{ 137 } & 1.4e-2 & 0.28 s & 1377 & 1.4e-3\\
\cdashline{2-2}
 & RK4 & 0.1 s & 476 & 4.1e-3 & 0.11 s & 476 & 4.1e-3 & 0.1 s & \textbf{ 476 } & 4.1e-3\\
\hline
\multirow{7}{*}{63} & Euler 1 & 0.24 s & 3987 & 4.9e-4 & 0.97 s & 16447 & 1.2e-4 & N.A. & 1391653 & (\textit{1.4e-6})\\
 & Lu--Cao 1 & 0.47 s & 2343 & 8.4e-4 & 2 s & 9668 & 2.0e-4 & N.A. & 907611 & (\textit{2.2e-6})\\
 & QC-1 & 18 ms & 80 & 2.5e-2 & 1.5 s & 8098 & 2.4e-4 & N.A. & 768994 & (\textit{2.6e-6})\\
\cdashline{2-2}
 & Euler 2 & 0.47 s & 3987 & 4.9e-4 & 0.46 s & 3987 & 4.9e-4 & \textbf{ 0.46 s } & 3987 & 4.9e-4\\
 & Lu--Cao 2 & 1.3 s & 2343 & 8.4e-4 & 1.3 s & 2343 & 8.4e-4 & 1.2 s & 2343 & 8.4e-4\\
 & QC-2 & \textbf{ 16 ms } & \textbf{ 19 } & 1.0e-1 & \textbf{ 76 ms } & \textbf{ 137 } & 1.4e-2 & 0.73 s & \textbf{ 1377 } & 1.4e-3\\
\cdashline{2-2}
 & RK4 & 1.2 s & 1963 & 1.0e-3 & 1.2 s & 1963 & 1.0e-3 & 1.2 s & 1963 & 1.0e-3\\
\hline
\multirow{7}{*}{127} & Euler 1 & 3.7 s & 16447 & 1.2e-4 & 3.9 s & 16447 & 1.2e-4 & N.A. & 1391653 & (\textit{1.4e-6})\\
 & Lu--Cao 1 & 9.2 s & 9668 & 2.0e-4 & 8.9 s & 9668 & 2.0e-4 & N.A. & 907611 & (\textit{2.2e-6})\\
 & QC-1 & 89 ms & 80 & 2.5e-2 & 6.9 s & 8098 & 2.4e-4 & N.A. & 768994 & (\textit{2.6e-6})\\
\cdashline{2-2}
 & Euler 2 & 7.8 s & 16447 & 1.2e-4 & 7.7 s & 16447 & 1.2e-4 & 7.7 s & 16447 & 1.2e-4\\
 & Lu--Cao 2 & 24 s & 9668 & 2.0e-4 & 25 s & 9668 & 2.0e-4 & 26 s & 9668 & 2.0e-4\\
 & QC-2 & \textbf{ 59 ms } & \textbf{ 19 } & 1.0e-1 & \textbf{ 0.35 s } & \textbf{ 137 } & 1.4e-2 & \textbf{ 3.5 s } & \textbf{ 1377 } & 1.4e-3\\
\cdashline{2-2}
 & RK4 & 20 s & 8098 & 2.4e-4 & 19 s & 8098 & 2.4e-4 & 18 s & 8098 & 2.4e-4\\
\hline
\end{tabular}

\caption{Computation time required to achieve a given precision, assuming the optimal $\Delta t$ is known. We use the lower plot of \cref{fig:simus}, where 40 different $\Delta t$ values have been tested, selecting the first one that yields an error below the required precision.
	For high precision with first-order schemes, values in parentheses indicate the expected required time-step in cases where the simulation would take too long to run; these values are extrapolated from the slopes of the corresponding curves in the lower plot of \cref{fig:simus}.
	The minimum computation time among the seven selected schemes is highlighted in \textbf{bold}, along with the smallest required number of steps. }
\label{table-cost_time}
\end{table}

\section{Conclusion and perspectives}

We have introduced and investigated a new class of Completely Positive Trace-Preserving schemes for the time discretization of  Lindblad master equations. We established that these schemes approximate the exact solution even when the Lindbladian is unbounded. Furthermore, through simple examples and numerical simulations, we showed that these CPTP schemes are robust and not subject to a CFL condition. This advantageous property is not shared by existing non-linear positive trace-preserving schemes.

We believe that the following avenues are worth exploring for future research:
\begin{enumerate}

    \item \textbf{Higher-order schemes.} Higher-order quantum channel schemes can be obtained using the method developed in \cite{caoStructurePreservingNumericalSchemes2024}. However, we have not yet performed an analysis of these schemes in the infinite-dimensional setting. Another interesting question would be to determine the minimum number of matrix operations required per time-step as a function of the order of the scheme. Following \cite{caoStructurePreservingNumericalSchemes2024}, the scaling
	    of the positivity-preserving schemes presented here is $\Theta(\Ndissip^p)$, where $p$ is the desired order and $N_d$ the number of dissipators entering the Lindblad equation, compared to $\Theta(d \Ndissip)$ for Euler or Runge-Kutta methods. For $p=2$, it is possible to achieve a scaling of $\Theta(\Ndissip)$; we are still investigating whether this observation could be generalized. This optimization is important for implementing fast, high-order schemes, especially when many jump operators are involved.

    \item \textbf{Time-dependent Lindbladian.} While we believe that most of our results can be extended to reasonably time-dependent Lindbladians with some technical assumptions, we also think that investigating possible optimizations of the numerical implementation would be valuable. Since $\oS = \sum_j \oM_j^\dag \oM_j$ is now time-dependent, the naive generalization of our algorithm would involve computing the inverse of its square root at each time-step. In the noteworthy case where the Hamiltonian is time-dependent but the dissipators are constant, the aforementionned difficulty is enterily circumvented since the matrix $\oS$ does not depend on $\oH$ (this fundamentally stems from the splitting operated in the definition of $\oM_{0,\Delta t}$).

    \item \textbf{Convergence of Galerkin approximations and link with time discretization.} This article does not investigate the convergence of the Galerkin approximation, that is the solution of \cref{eq-Lindblad_truncated}, to the solution of the original Lindblad equation. A recent preprint \cite{etienneyPosterioriErrorEstimates2025} provides \textit{a posteriori} error estimates, but does not include a proof of convergence. We hypothesise that the \textit{a priori} estimates presented in this article could be leveraged to establish convergence with an explicit rate. Once this is achieved, an interesting next step would be to generalize the results of \cref{subsec:cvg_ex_cptp} to dynamics that lack the stability property described in \cref{rmk_stability_photon_loss}. Specifically, this would involve demonstrating that for the CPTP schemes introduced in this paper, the error between the exact solution and the time discretized solution of the Galerkin approximation can be bounded by the sum of two terms: one controlling the time discretization error and the other controlling the Galerkin error.
\end{enumerate}
\subsection*{Acknowledgments}

The authors would like to express their gratitude to Claude Le Bris for invaluable fruitful discussions,
and to Ronan Gautier, Pierre Guilmin, Mazyar Mirrahimi, Alexandru Petrescu, Alain Sarlette, and Antoine Tilloy for their insightful feedback.

This project has received funding from the European Research Council (ERC) under the European Union’s Horizon 2020 research and innovation program (grant agreement No. 884762).
Part of this research was performed while the first and last authors were visiting the Institute for Mathematical and Statistical Innovation (IMSI) in Chicago, which is supported by the National Science Foundation (Grant No. DMS-1929348).

\begin{appendix}
 
\section{Proof of Lemma \ref{prop_regularity}
}
\label{sec_proof_prop_regularity_sans_H}

    Let us establish the \textit{a priori} estimate stated in \cref{eq_apriori_withoutH}. This proof is inspired by \cite[Section 3.6]{fagnolaQuantumMarkovSemigroups1999}.
    Under \cref{assump_sansH}, for every $k \geq 0$, $\D(\oLambda^k)=\xH^{2k}= \D(\oG^k)$ is stable under the semigroup $(e^{t\oG})$. In \cref{assump_apriori} we assumed the stability of $\xH^{4}$ and $\xH^{6}$.
    In both cases, \cref{apriori_sans_H} ensures that it is $\gamma$-quasi-dissipatif on $\xH^{4}$.
    
    For $\oLambda >0$, let us consider $R^{(m)}_\oLambda$, which denotes the truncation of the series characterizing the resolvent $R^{min}_\lambda$ of the minimal semigroup (refer to \cref{eq_resolvent_min}):
    \begin{align}
        R_\lambda^{(m)}(\oX)=\sum_{k=0}^m Q_\lambda^k(P_\lambda(\oX)), \quad \oX\in B(\xH).
    \end{align}
    Let us show by induction in $m$ that for any $m,n\in \N$, $\lambda >\gamma, u \in \xH^4$
    \begin{align}
        (\lambda-\gamma)\braket{u|R_\lambda^{(m)}(n\oLambda^4R(n,-\oLambda^4))u}\leq \|\oLambda^2u\|^2
    \end{align}
    For $m=0$ we have and
    \begin{align*}
        \braket{u|R_\lambda^{(0)}(n\oLambda^4R(n,-\oLambda^4))u}
        &=\int_0^\infty e^{-\lambda s} \braket{e^{s\oG}u|n\oLambda^4R(n,-\oLambda^4) e^{s\oG}  u}ds\\
        & \leq \int_0^\infty e^{-\lambda s} \|\oLambda^2 e^{s\oG}u\|^2\\
        &\leq  \int_0^\infty e^{(\gamma-\lambda) s} \|\oLambda^2 u\|^2= \frac{1}{\lambda- \gamma} \|\oLambda^2 u\|^2 
    \end{align*}
    Next, using $R^{(m+1)}_\lambda= Q_\lambda(R^{(m)}_\lambda)+ P_\lambda$, we get for every $u \in \xH^6$
    \begin{align}
        &\braket{u|R_\lambda^{(m+1)}(n\oLambda^4R(n,-\oLambda^4))u}\\
        &=\braket{u|P_\lambda (n\oLambda^4R(n,-\oLambda^4))u}+\sum_{j=1}^{\Ndissip} \int_0^\infty e^{-\lambda s} \braket{\oL_j e^{s\oG}u|R_\lambda^{(m)}(n\oLambda^4R(n,-\oLambda^4))\oL_j e^{s\oG} u}ds\\
        &\leq \int_0^\infty e^{-s\lambda}\|\oLambda^2 e^{s\oG} u\|^2ds+\frac{1}{\lambda-\gamma }\sum_{j=1}^{\Ndissip} \int_0^\infty e^{-\lambda s} \|\oLambda^2 \oL_j e^{s\oG}u\|^2ds
    \end{align}
    where we used the induction and the fact $\xH^6$ is stable under $(e^{s\oG})$. We continue using \cref{apriori_sans_H}
    \begin{align}
        \label{eq_979}
        &\leq \int_0^\infty e^{-s\lambda}\|\oLambda^2 e^{s\oG} u\|^2ds+\frac{1}{\lambda-\gamma }\int_0^\infty e^{-\lambda s} (\gamma\|\oLambda^2 e^{s\oG}u\|^2-2\operatorname*{Re}\braket{\oLambda^2 \oG e^{s\oG} u|\oLambda^2 e^{s\oG} })ds\\
        &\leq \frac{\lambda}{\lambda-\gamma}\int_0^\infty e^{-s\lambda}\|\oLambda^2 e^{s\oG} u\|^2ds-\frac{2}{\lambda-\gamma }\int_0^\infty e^{-\lambda s}\operatorname*{Re}\braket{\oLambda^2 \oG e^{s\oG} u|\oLambda^2 e^{s\oG} }ds
    \end{align}
    Performing an integration by part on the right term of the previous sum gives
    \begin{align}
        -\frac{2}{\lambda-\gamma}\int_0^\infty e^{-\lambda s}\operatorname*{Re}\braket{\oLambda^2 \oG e^{s\oG} u|\oLambda^2 e^{s\oG} }ds&=
        \frac{-1}{\lambda-\gamma}\int_0^\infty e^{-\lambda s}\frac{d}{ds} \left(\|\oLambda^2 e^{s\oG} u\|^2 \right) ds\\
        &=\frac{1}{\lambda-\gamma}\|\oLambda^2u\|^2 -\frac{\lambda}{\lambda-\gamma} \int_0^\infty e^{-s\lambda}\|\oLambda^2 e^{s\oG} u\|^2ds.
    \end{align}
    Reinjecting in \cref{eq_979} leads to
    \begin{align}
        \braket{u|R_\lambda^{(m+1)}(n\oLambda^4R(n,-\oLambda^4))u}
        \leq \frac{1}{\lambda-\gamma}\|\oLambda^2u\|^2,
    \end{align}
    for all $u \in \xH^6$. As $\xH^6$ is a core for $\oLambda^2$, the previous equality holds for any $u \in \xH^4$ Then taking the supremum on $m$ gives
    \begin{align}
        \label{eq_resolvant_estimate}
        \braket{u|R_\lambda^{min}(n\oLambda^4R(n,-\oLambda^4))u}
        \leq \frac{1}{\lambda-\gamma}\|\oLambda^2u\|^2, \quad u\in \xH^4.
    \end{align}
    Then, using that 
    \begin{align}
        T_t^{min}(\oX) =s-\lim_{k\to \infty} (\frac{k}{t} R_{\frac{k}{t}}^{min})^k \oX, \qquad \oX\in B(\xH),
    \end{align}
    we get 
    \begin{align}
        \braket{u|\T_t^{min}(n\oLambda^4R(n,-\oLambda^4))u}
        \leq \|\oLambda^2u\|^2 \limsup_{k\to \infty} (\frac{k/t}{k/t-\gamma})^k=e^{\gamma t}\|\oLambda^2u\|^2,
    \end{align}
    which concludes the proof of \cref{eq_apriori_withoutH}.
    It has been proven in \cite{chebotarevLindbladEquationUnbounded1997,chebotarevSufficientConditionsConservativity1998}, that (a weaker version of) the \textit{a priori} estimate \cref{eq_apriori_withoutH} implies the conservativity of the minimal semigroup. 
    
    \begin{comment}
    \rrwarning{réécrire cette partie}
    To prove the conservativity of the minimal semigroup, we use part of the result obtained by Chebotarev and Fagnola in \cite{chebotarevSufficientConditionsConservativity1998}.
    Namely, we introduce 
    \begin{align}
        F_n&=\sum_{j=1}^{\Ndissip} L_j^\dag L_j nR(n,-\oLambda^4)
    \end{align}
    then, we have using that $\oLambda$ and $G$ commute that
    \begin{align*}
        P_\lambda(F_n)= Q_\lambda(\Id) nR(n,-\oLambda^4)
    \end{align*}
    Then, following \cite[Prop 3.34]{chebotarevSufficientConditionsConservativity1998}
    \begin{align}
        \sum_{k=0}^\infty \braket{u,Q_\lambda^{k+1}(\Id)u}&
        =\sup_n \sum_{k=0}^\infty \braket{u|Q_\lambda^{k}(F_n)u}\\
        &=\sup_n \braket{u| R^{min}_{\lambda}(F_n)}.
    \end{align}
    As $F_n\leq \oLambda R(n,-\oLambda^4) \leq \oLambda^4 R(n,-\oLambda^4)$,
    \cref{eq_resolvant_estimate} implies that the series of positive terms on left-hand side is bouded for every $u \in \xH^4$. Thus, $s-\lim_{n\to \infty} Q_\lambda^n(\Id)=0$ which implies conservativity
    \rrnote{expliquer un peu plus en utilisant 3.40 de Fagnola}.
\end{comment}

\section{Interpolation}
\label{app_interpol}

\begin{lemma}
    \label{lem_interpol}
    Let $n_0 \geq l \geq 0$ and $\oT$ be a continuous operator from $\xH^l\to \xH^0$ and from $\xH^{n_0}\to \xH^{n_0-l}$.
    Then the restriction of $\oT$ to $\xH^{n_0}$ can be extended to a continuous operator $\xH^{n}\to \xH^{n-l}$ for all $n \in [n_0,l]$.
\end{lemma}
\begin{proof}
    This is a standard result of interpolation theory. For every $\ket{\psi}\in \xH^{n_0},\, \ket{\varphi} \in \xH$, we introduce the function $g$ defined on the strip $S=\{ z \in \C \mid 0 \leq \operatorname*{Re}(z)\leq 1\}$
\begin{align*}
    g(z)= \bra{\varphi} \oLambda^{z(n_0-l)/2} \oT \oLambda^{-z(n_0-l)/2-l/2} \ket{\psi}
\end{align*}
    Let us show that $g$ is continuous on $S$ and analytic on $S^\circ$. We rewrite the previous equation as
    \begin{align}
        g(z)=\bra{\varphi} \oLambda^{(z-1)(n_0-l)/2} \oLambda^{(n_0-l)/2} \oT \oLambda^{-n_0/2} \oLambda^{-z(n_0-l)/2-l/2} \oLambda^{n_0/2} \ket{\psi}
    \end{align}
    We recall that $\oLambda^{(z-1)(n_0-l)/2}$, $\oLambda^{(n_0-l)/2} \oT \oLambda^{-n_0/2}$ and $\oLambda^{-z(n_0-l)/2-l/2}$ are bounded on $\xH$. Besides
    \begin{align}
        z \mapsto \bra{\varphi}\oLambda^{(z-1)(n_0-l)/2}, \quad z \mapsto \oLambda^{-z(n_0-l)/2-l/2} \oLambda^{n_0/2} \ket{\psi}
    \end{align}
     are continuous on $S$ and holomorphic (As $\oLambda$ is the generator of an analytic semigroup) on $S^\circ$; thus so is $g$.

    Let us now compute
    \begin{align}
        |g(it)|&=|\bra{\varphi}\oLambda^{itn_0/2}\oT \oLambda^{-itn_0/2-l/2} \ket{\psi}|\\
        &\leq \|\oT\|_{\xH^l \to \xH} \|\ket{\varphi}\| \|\ket{\psi}\|
    \end{align}
    and 
    \begin{align}
        |g(1+it)|&=|\bra{\varphi}\oLambda^{itn_0/2+(n_0-l)/2}\oT \oLambda^{-itn_0/2-n_0/2} \ket{\psi}|\\
        &\leq \|\oT\|_{\xH^{n_0} \to \xH^{n_0-l}} \|\ket{\varphi}\| \|\ket{\psi}\|
    \end{align}
    We deduce using Hadamard’s three-lines theorem that for $0 \leq \theta \leq 1$ 
    \begin{align}
        |g(\theta)|\leq \|\ket{\varphi}\| \|\ket{\psi}\|  \|\oT\|_{\xH^{n_0} \to \xH^{n_0-l}}^{1- \theta} \|\oT\|_{\xH^l \to \xH}^\theta
    \end{align}
    Taking the supremum over $\ket{\varphi}$ such that $\|\ket{\varphi}\|=1$ leads to
    \begin{align}
        \|\oLambda^{(\theta n_0-l)/2} \oT \oLambda^{-\theta}\ket{\psi}\| \leq \| \ket{\psi} \| \|\oT\|_{\xH^{n_0} \to \xH^{n_0-l}}^{1- \theta} \|\oT\|_{\xH^l \to \xH}^\theta
    \end{align}
    Thus $\oT$ can be extended to a bounded linear map $\xH^{\theta n_0}\to \xH^{\theta n_0-l}$.
\end{proof}

%\paragraph{The semigroup $e^{tG}$}
%There exists $w_1,w_2 >0$ such that $\forall \ket{\psi}\in \D(\oLambda^5)$,
%        \begin{align}
%            \label{eq_dissipativity}
%            2 \operatorname*{Re} ( \braket{G \psi | \oLambda \psi}) \leq w_1 \|\oLambda^{1/2} \ket{\psi}\|^2\\
%            \label{eq_dissipativity_2}
%            2 \operatorname*{Re} ( \braket{G \psi | \oLambda^5 \psi}) \leq w_2 \|\oLambda^{5/2} \ket{\psi}\|^2
%        \end{align}
%First we consider $G=-iH -\frac{1}{2}\sum_k L_k^\dag L_k$ on $\D(\oLambda)$. Using relative boundedness with respect to $\oLambda$, we deduce that $D(\oLambda)$ is in the domain of its adjoint, hence $G,\D(\oLambda)$ is closable with domain $\D(G)$. As $G$ is dissative on $\xH$. To prove that it genererate a semigroup of contraction on $\xH$, it remains to show that $(G+ \Id)$ is surjective. We introduce $G_\epsilon=G-\epsilon \oLambda$ on $D(\oLambda)$. This operator is closed and generate a semigroup of contraction on both $\xH$ and $\xH^1$.
%Indeed
%\begin{align}
%    \langle G_\epsilon \psi, \varphi \rangle_1
%\end{align}
\end{appendix}

\printbibliography

\end{document}